\documentclass[reqno,11pt]{article} 
\usepackage[left=2.5cm,right=2.5cm,top=3cm,bottom=3cm,a4paper]{geometry}
\usepackage{amsmath, amssymb}
\usepackage{amsthm, amscd} 
\usepackage[all,cmtip]{xy}
\usepackage{float}
\usepackage{caption}
\usepackage{color}

\newcommand{\Mod}[1]{\ (\textup{mod}\ #1)}
\theoremstyle{plain} 
\newtheorem{theorem}{\indent\sc Theorem}[section]
\newtheorem{lemma}[theorem]{\indent\sc Lemma}

\newtheorem{proposition}[theorem]{\indent\sc Proposition}

\theoremstyle{definition} 
\newtheorem{definition}[theorem]{\indent\sc Definition}
\newtheorem{assumption}[theorem]{\indent\sc Assumption}

\newtheorem{remark}[theorem]{\indent\sc Remark}

\makeatletter
\def\address#1#2{\begingroup
\noindent\parbox[t]{7.8cm}{%
\small{\scshape\ignorespaces#1}\par\vskip1ex
\noindent\small{\itshape E-mail address}%
\/: #2\par\vskip4ex}\hfill%
\endgroup}%
\makeatother
\title{Form class groups and class fields of CM-fields}
\author{
\textsc{Ho Yun Jung, Ja Kyung Koo, Dong Hwa Shin and Dong Sung Yoon} 
}
\date{} 
\begin{document}

\allowdisplaybreaks

\maketitle

\footnote{ 
2010 \textit{Mathematics Subject Classification}. Primary 11E12; Secondary 11F41, 11R29, 11R37.}
\footnote{ 
\textit{Key words and phrases}. Binary quadratic forms, class field theory, CM-fields, Hilbert modular functions,
ideal class groups.} \footnote{
\thanks{
The first author was supported by the research fund of Dankook University in 2020.
The third (corresponding) author was supported by the Hankuk University of Foreign Studies Research Fund of 2020.
The fourth author was
supported by the National Research Foundation of Korea(NRF) grant funded by the Korea government(MSIT) (2020R1C1C1A01006139). 
}
}
\vspace{-1cm}
\begin{center}
Dedicated to the late professor Goro Shimura
\end{center}

\begin{abstract}
Let $F$ be a totally real number field of class number one,
and let $K$ be a CM-field with $F$ as its maximal real subfield.
For each positive integer $N$,
we construct a class group of certain binary quadratic forms over $F$
which is isomorphic to the ray class group of $K$ modulo $N$.
Assuming further that the narrow class number of $F$ is one,
we construct a class field of the reflex field of $K$
in terms of the singular values of Hilbert modular functions.
\end{abstract}

\maketitle
\tableofcontents

\section {Introduction}

In the \textit{Disquisitiones Arithmeticae} published in 1801 (\cite{Gauss}) Gauss introduced
the direct composition on the set of primitive positive definite binary quadratic forms of given (even) discriminant, although he did not
recognize the concept of a group.
Eventually, in 2004, Bhargava first discovered higher composition laws
and gave a simplification of Gauss' composition
in terms of certain cubes of integers
(\cite{BhargavaI}, \cite{BhargavaII}, \cite{BhargavaIII} and \cite{Buell}).
\par
Let $K$ be an imaginary quadratic field
and $\mathcal{O}$ be the order of discriminant $D$ in $K$.
And, let $\mathcal{Q}(D)$ be the set of
primitive positive definite binary quadratic forms
$Q(x,\,y)=ax^2+bxy+cy^2$ of discriminant $b^2-4ac=D$
on which the modular group $\mathrm{SL}_2(\mathbb{Z})$ defines the
proper equivalence as
\begin{equation*}
Q\sim Q'\quad\Longleftrightarrow\quad
Q'\left(\begin{bmatrix}x\\y\end{bmatrix}\right)=Q\left(\gamma\begin{bmatrix}x\\y\end{bmatrix}\right)~
\textrm{for some}~\gamma\in\mathrm{SL}_2(\mathbb{Z}).
\end{equation*}
For each $Q(x,\,y)\in\mathcal{Q}(d_K)$, we
let $\omega_Q$ be the zero of the quadratic polynomial $Q(x,\,1)$ lying in
the complex upper half-plane $\mathbb{H}$.
Then, the modern algebraic number theory due to Dedekind enables us to have a bijection
between the set of equivalence classes $\mathcal{C}(D)=\mathcal{Q}(D)/\sim$ and
the group $\mathcal{C}(\mathcal{O})$ of proper fractional $\mathcal{O}$-ideals,
namely,
\begin{equation*}
\begin{array}{ccc}
\mathcal{C}(D)&\rightarrow&\mathcal{C}(\mathcal{O})\\
\mathrm{[}Q\mathrm{]}&\mapsto&
\mathrm{[}\mathrm{[}\omega_Q,\,1\mathrm{]}\mathrm{]}
\end{array}
\end{equation*}
where $[\omega_Q,\,1]=\mathbb{Z}\omega_Q+\mathbb{Z}$.
Through this bijection we endow $\mathcal{C}(D)$ with the so-called Dirichlet composition
so that it becomes an abelian group
(\cite[(3.7) and Theorem 7.7]{Cox}).
On the other hand, we know by the theory of complex multiplication on elliptic curves
completed by Hasse that the ring class field
$H_\mathcal{O}$ of order $\mathcal{O}$ is generated by the $j$-invariant $j(\mathcal{O})$
of the elliptic curve $\mathbb{C}/\mathcal{O}$.
Furthermore, there is an isomorphism
\begin{equation*}
\begin{array}{ccc}
\mathcal{C}(\mathcal{O})&\stackrel{\sim}{\rightarrow}&\mathrm{Gal}(H_\mathcal{O}/K)\\
\mathrm{[}\mathfrak{a}\mathrm{]}&\mapsto&(j(\mathcal{O})\mapsto j(\overline{\mathfrak{a}}))
\end{array}
\end{equation*}
(\cite{Hasse} or \cite[Theorem 11.1 and Corollary 11.37]{Cox}), and hence we obtain an isomorphism
\begin{equation}\label{CDGHK}
\begin{array}{ccc}
\mathcal{C}(D)&\stackrel{\sim}{\rightarrow}&\mathrm{Gal}(H_\mathcal{O}/K)\\
\mathrm{[}Q\mathrm{]}&\mapsto&\left(j(\mathcal{O})\mapsto j([-\overline{\omega_Q},\,1])\right).
\end{array}
\end{equation}
Here, $\overline{\,\cdot\,}$ stands for the complex conjugation.
\par
Before Hasse, Hecke constructed unramified abelian extensions of some biquadratic fields
by means of singular values of Hilbert modular functions of two variables analogous to
the $j$-function of the elliptic case (\cite{Hecke1912}, \cite{Hecke1913} and \cite{Schappacher}).
Later, Shimura succeeded in constructing unramified or ramified abelian extensions of
certain CM-fields through the investigation on abelian varieties
of the same type with complex multiplication and polarizations, and also the analysis of the points of finite order
on abelian varieties (\cite{Shimura} and \cite{S-T}).
On the other hand,
Zemkova recently gave a correspondence between binary quadratic forms over any number field $F$ of narrow class number one
and a certain ideal class group of a relative quadratic extension of $F$ (\cite{Zemkova}). See also \cite{Mastropietro}
for the special case where $F$ is real quadratic.
In this paper, inspired by these results,
we intend to extend the classical isomorphism in (\ref{CDGHK}) to
a general CM-field.
\par
More precisely, we consider a CM-field $K$ whose maximal real subfield $F$ is of class number one.
Denote by $\mathcal{O}_K$ and $\mathcal{O}_F$ the rings of algebraic integers in $K$ and $F$, respectively. Since $\mathcal{O}_F$ is a PID,
for a pair of elements $a$ and $b$ of $\mathcal{O}_F$ not both zero,
one can consider their greatest common divisor
$\gcd(a,\,b)$ up to associate
(\cite[Definition 2.2 in Chapter V]{Aluffi}).
And we take an element $\omega_K$ of $K\cap\mathbb{H}$ such that
$\mathcal{O}_K=\mathcal{O}_F\omega_K+\mathcal{O}_F$
(Lemma \ref{basis}), and let
$d_K$ ($<0$) be the discriminant of the minimal polynomial of $\omega_K$ over $F$.
For a positive integer $N$, let
\begin{equation*}
\mathcal{Q}_F(N,\,d_K)=\{ax^2+bxy+cy^2\in\mathcal{O}_F[x,\,y]~|~
\gcd(a,\,b,\,c)=1,~a>0,\,b^2-4ac=d_K,\,\gcd(a,\,N)=1\}.
\end{equation*}
Then the arithmetic subgroup
\begin{equation*}
\Gamma_{F,\,1}(N)=\left\{
\gamma=\begin{bmatrix}c_1&c_2\\c_3&c_4\end{bmatrix}
\in\mathrm{GL}_2(\mathcal{O}_F)~|~\det(\gamma)>0~\textrm{and}~
c_3\equiv0,\,
c_4\equiv\zeta\Mod{N\mathcal{O}_F}~\textrm{for some}~\zeta\in\mathcal{O}_F^\times\right\}
\end{equation*}
of $\mathrm{GL}_2(F)$ defines an equivalence relation $\sim_{\Gamma_{F,\,1}(N)}$ on the set $\mathcal{Q}_F(N,\,d_K)$ as
follows:
\begin{equation*}
Q\sim_{\Gamma_{F,\,1}(N)} Q'\quad
\Longleftrightarrow\quad Q'=\frac{1}{\det(\gamma)}Q\left(\gamma\begin{bmatrix}x\\y\end{bmatrix}\right)
~\textrm{for some}~\gamma\in\Gamma_{F,\,1}(N)
\end{equation*}
(Lemma \ref{action} and Definition \ref{equivalencerelation}).
As for the arithmetic subgroup, we refer to \cite[$\S$24.8]{Shimura}.
Now, as the first main theorem of this paper, we shall explain that
the set of equivalence classes
\begin{equation*}
\mathcal{C}_F(N,\,d_K)=\mathcal{Q}_F(N,\,d_K)/\sim_{\Gamma_{F,\,1}(N)}
\end{equation*}
can be viewed as a group isomorphic to the ray class group $\mathcal{C}(N\mathcal{O}_K)$ of $K$
modulo $N\mathcal{O}_K$ by verifying that the natural mapping
\begin{equation*}
\begin{array}{ccc}
\mathcal{C}_F(N,\,d_K)&\mapsto&\mathcal{C}(N\mathcal{O}_K)\\
\mathrm{[}Q\mathrm{]}&\mapsto&\mathrm{[}\mathrm{[}\omega_Q,\,1\mathrm{]}_F\mathrm{]}
\end{array}
\end{equation*}
is bijective, where $\omega_Q$ is the zero of $Q(x,\,1)$ lying in $\mathbb{H}$
and $[\omega_Q,\,1]_F=\mathcal{O}_F\omega_Q+\mathcal{O}_F$
(Theorem \ref{main1}).
\par
Next, the second main theorem exhibits the Hilbert 12th problem
generating class fields in terms of singular values. To be precise,
let $(K,\,\{\varphi_i\}_{i=1}^g)$ be a CM-type such that
\begin{equation*}
\varphi_1=\mathrm{id}_K\quad\textrm{and}\quad
\varphi_i(\omega_K)\in\mathbb{H}\quad(i=1,\,2,\,\ldots,\,g)
\end{equation*}
with $K^*$ its reflex field.
Assume further that the narrow class number of $F$ is one.
For the arithmetic subgroup
\begin{equation*}
\Gamma_{F,\,1}^+(N)=\{\gamma\in\Gamma_{F,\,1}(N)~|~\det(\gamma)\gg0\}
\end{equation*}
of $\mathrm{GL}_2(F)$, let $\mathcal{A}_0(\Gamma_{F,\,1}^+(N),\,\mathbb{Q})$ be
the field of Hilbert modular functions for $\Gamma_{F,\,1}^+(N)$ with rational Fourier coefficients.
And, we set
\begin{equation*}
\mathbf{w}=
(\varphi_1(-\overline{\omega_K}),\,
\varphi_2(-\overline{\omega_K}),\,\ldots,\,
\varphi_g(-\overline{\omega_K}))
\end{equation*}
as a CM-point on $\mathbb{H}^g$ in the sense of \cite[$\S$24.10]{Shimura}.
Then we shall prove that the field
\begin{equation*}
L=K^*(f(\mathbf{w})~|~f\in\mathcal{A}_0(\Gamma_{F,\,1}^+(N),\,\mathbb{Q})~
\textrm{is finite at}~\mathbf{w})
\end{equation*}
is a subfield of the ray class field $K^*_{(N)}$ of $K^*$ modulo $(N)=N\mathcal{O}_{K^*}$ with
\begin{equation*}
\mathrm{Gal}(L/K^*)\simeq\mathcal{C}(N\mathcal{O}_{K^*})/\mathrm{Ker}(\mathfrak{g}_N)
\end{equation*}
where $\mathfrak{g}_N:\mathcal{C}(N\mathcal{O}_{K^*})
\rightarrow\mathcal{C}(N\mathcal{O}_K)$ is the homomorphism induced from the reflex norm map
(Theorem \ref{main2}).
To this end, we shall use the theory of complex multiplication
of higher dimensional abelian varieties developed by Shimura, Taniyama
and Weil.
\par
On the other hand, it is worthy of note that Ribet showed
by utilizing the theory of modular forms and Galois representations
that for each irregular prime $p$ there exists an unramified abelian extension
of the $p$th cyclotomic field $\mathbb{Q}(e^{2\pi\mathrm{i}/p})$
whose Galois group is killed by $p$
(\cite{Mazur} and \cite{Ribet}).
Besides, over real quadratic fields Darmon, Lauder and Rotger suggested
an approach to explicit class field theory
by making use of techniques from $p$-adic modular forms and
deformations of Galois representations (\cite{D-L-R}).

\section {Binary quadratic forms over the maximal real subfield}

Throughout this paper, we let $K$ be a CM-field with maximal real subfield $F$.
This means that $K$ is a totally imaginary quadratic extension of a totally real number field $F$.
By $\mathcal{O}_F^\times$ we mean the group of units in the ring of integers $\mathcal{O}_F$ of $F$.

\begin{assumption}\label{classnumberone}
We assume that the class number $h_F$ of $F$ is one so that
$\mathcal{O}_F$ is a PID.
\end{assumption}

Let $\mathcal{Q}_F$ be the set of
primitive positive definite binary quadratic forms over $\mathcal{O}_F$.
That is,
$Q(x,\,y)=ax^2+bxy+cy^2$ ($\in\mathcal{O}_F[x,\,y]$)
belongs to $\mathcal{Q}_F$ if and only if
\begin{equation*}
\gcd(a,\,b,\,c)=1,\quad a>0\quad\textrm{and}\quad
d_Q<0,
\end{equation*}
where $d_Q=b^2-4ac$ is the discriminant of $Q$.
For $\nu_1,\nu_2\in K$, we write
\begin{equation*}
[\nu_1,\,\nu_2]_F=\mathcal{O}_F\nu_1+\mathcal{O}_F\nu_2.
\end{equation*}
Although the following lemma appeared in \cite[Proposition 1.1]{Zemkova},
we shall give its proof for the sake of completeness.

\begin{lemma}\label{basis}
There is an element $\omega$ of $K\cap\mathbb{H}$ such that
$\mathcal{O}_K=[\omega,\,1]_F$.
\end{lemma}
\begin{proof}
By Assumption \ref{classnumberone} and the fact $[K:F]=2$, we obtain
\begin{equation}\label{oplus}
\mathcal{O}_K=\mathcal{O}_F\omega_1\oplus\mathcal{O}_F\omega_2
\quad
\textrm{for some}~\omega_1,\,\omega_2\in\mathcal{O}_K
\end{equation}
(\cite[Corollary to Proposition 7.46]{Narkiewicz}).
By changing $\omega_1$ and
$\omega_2$ if necessary, we may assume
\begin{equation*}
\omega=\overline{\omega_1}\omega_2\in\mathbb{H}.
\end{equation*}
We then see that
\begin{eqnarray*}
\Delta(\mathcal{O}_K/\mathcal{O}_F)&=&\left|\begin{matrix}
T_{K/F}(\omega_1\cdot\omega_1) &
T_{K/F}(\omega_1\cdot\omega_2)\\
T_{K/F}(\omega_2\cdot\omega_1) &
T_{K/F}(\omega_2\cdot\omega_2)\end{matrix}\right|
\mathcal{O}_F\quad\textrm{by (\ref{oplus})}\\
&=&\left|\begin{matrix}
\omega_1^2+\overline{\omega_1}^2 & \omega_1\omega_2+\overline{\omega_1\omega_2}\\
\omega_2\omega_1+\overline{\omega_2\omega_1} & \omega_2^2+\overline{\omega_2}^2
\end{matrix}\right|\mathcal{O}_F\\
&=&(\overline{\omega_1}\omega_2-\omega_1\overline{\omega_2})^2\mathcal{O}_F\\
&=&(\omega-\overline{\omega})^2\mathcal{O}_F
\end{eqnarray*}
(\cite[Lemma 7.2 in Chapter I]{Janusz}) and
\begin{eqnarray*}
d_{K/F}(\omega,\,1)=\left|\begin{matrix}
T_{K/F}(\omega\cdot\omega) & T_{K/F}(\omega\cdot1)\\
T_{K/F}(1\cdot\omega) & T_{K/F}(1\cdot1)
\end{matrix}\right|=\left|\begin{matrix}
\omega^2+\overline{\omega}^2 & \omega+\overline{\omega}\\
\omega+\overline{\omega} & 1+1
\end{matrix}\right|=(\omega-\overline{\omega})^2.
\end{eqnarray*}
Hence it follows that $\Delta(\mathcal{O}_K/\mathcal{O}_F)=d_{K/F}(\omega,\,1)\mathcal{O}_F$,
which implies by \cite[Proposition 2.24]{Milne}
\begin{equation*}
\mathcal{O}_K=\mathcal{O}_F\omega\oplus\mathcal{O}_F\quad\textrm{with}~\omega\in\mathbb{H}.
\end{equation*}
\end{proof}

From now on, we fix such an element $\omega$ in Lemma \ref{basis} and denote it by $\omega_K$.
Let $d_K$ be the discriminant of the minimal polynomial of $\omega_K$ over $F$, namely,
if
\begin{equation*}
\min(\omega_K,\,F)=x^2+b_Kx+c_K\quad(\in\mathcal{O}_F[x]),
\end{equation*}
then
$d_K=b_K^2-4c_K$. Observe that $d_K$ is totally negative because $F$ is totally real and $K=F(\omega_K)$ is totally imaginary.
Furthermore, for each $Q(x,\,y)=ax^2+bxy+cy^2\in\mathcal{Q}_F$, let
$\omega_Q$ be the zero of the quadratic polynomial $Q(x,\,1)$ in $\mathbb{H}$,
that is,
\begin{equation*}
\omega_Q=\frac{-b+\sqrt{d_Q}}{2a}.
\end{equation*}

For a positive integer $N$, we consider
the subset $\mathcal{Q}_F(N,\,d_K)$ of $\mathcal{Q}_F$ given by
\begin{equation*}
\mathcal{Q}_F(N,\,d_K)=\{Q=ax^2+bxy+cy^2\in\mathcal{Q}_F~|~\gcd(a,\,N)=1~
\textrm{and}~d_Q=d_K\}.
\end{equation*}
Let $\Gamma_{F,\,1}(N)$ be a subgroup of $\mathrm{GL}_2(\mathcal{O}_F)$ defined by
\begin{equation*}
\Gamma_{F,\,1}(N)=\left\{
\gamma=\begin{bmatrix}c_1&c_2\\c_3&c_4\end{bmatrix}
\in\mathrm{GL}_2(\mathcal{O}_F)~|~\det(\gamma)>0~\textrm{and}~
c_3\equiv0,\,
c_4\equiv\zeta\Mod{N\mathcal{O}_F}~\textrm{for some}~\zeta\in\mathcal{O}_F^\times\right\}
\end{equation*}
which is an arithmetic subgroup of $\mathrm{GL}_2(F)$.

\begin{lemma}\label{action}
The group $\Gamma_{F,\,1}(N)$ acts on the set $\mathcal{Q}_F(N,\,d_K)$ as
\begin{equation*}
Q^\gamma\left(\begin{bmatrix}x\\y\end{bmatrix}\right)=
\frac{1}{\det(\gamma)}Q\left(\gamma\begin{bmatrix}x\\y\end{bmatrix}\right)\quad
(Q\in\mathcal{Q}_F(N,\,d_K),~\gamma\in\Gamma_{F,\,1}(N)).
\end{equation*}
\end{lemma}
\begin{proof}
Let $Q=ax^2+bxy+cy^2\in\mathcal{Q}_F(N,\,d_K)$, $\gamma=\begin{bmatrix}c_1&c_2\\c_3&c_4\end{bmatrix}$ and $\sigma\in\Gamma_{F,\,1}(N)$.
Set
$Q'=a'x^2+b'xy+c'y^2=Q^\gamma$.
Note first that
\begin{equation*}
d_{Q'}=\frac{1}{\det(\gamma)^2}\cdot d_Q\cdot\det(\gamma)^2=d_Q=d_K.
\end{equation*}
We also see that
\begin{equation*}
a'=\frac{1}{\det(\gamma)}(ac_1^2+bc_1c_3+cc_3^2)=\frac{1}{\det(\gamma)}Q(c_1,\,c_3)>0
\end{equation*}
because $\det(\gamma)>0$ and $Q$ is positive definite. Moreover, since $c_3\equiv0\Mod{N\mathcal{O}_F}$,
$a'$ satisfies
\begin{equation*}
a'\equiv\frac{1}{\det(\gamma)}ac_1^2\Mod{N\mathcal{O}_F},
\end{equation*}
which shows that $\gcd(a',\,N)=1$.
If $Q'$ is not primitive, then the relation
\begin{equation*}
Q\left(\begin{bmatrix}x\\y\end{bmatrix}\right)=
\det(\gamma)Q'\left(\gamma^{-1}\begin{bmatrix}x\\y\end{bmatrix}\right)
\end{equation*}
implies that $Q$ is not primitive, which yields a contradiction. Thus,
$Q'$ belongs to $\mathcal{Q}_F(N,\,d_K)$.
We also establish that
$Q^{I_2}=Q$ and
\begin{eqnarray*}
(Q^\gamma)^\sigma&=&Q'^\sigma\\
&=&\frac{1}{\det(\sigma)}Q'\left(\sigma\begin{bmatrix}x\\y\end{bmatrix}\right)\\
&=&\frac{1}{\det(\sigma)}\cdot\frac{1}{\det(\gamma)}Q\left(\gamma\sigma\begin{bmatrix}
x\\y\end{bmatrix}\right)\\
&=&\frac{1}{\det(\gamma\sigma)}Q\left(\gamma\sigma\begin{bmatrix}
x\\y\end{bmatrix}\right)\\
&=&Q^{\gamma\sigma}.
\end{eqnarray*}
Therefore, $\Gamma_{F,\,1}(N)$ gives a well-defined action on $\mathcal{Q}_F(N,\,d_K)$.
\end{proof}

\begin{remark}\label{actionzero}
Let $Q\in\mathcal{Q}_F(N,\,d_K)$ and $\gamma=\begin{bmatrix}c_1&c_2\\c_3&c_4\end{bmatrix}\in\Gamma_{F,\,1}(N)$. Since
\begin{equation*}
0=Q^\gamma\left(\begin{bmatrix}\omega_{Q^\gamma}\\1\end{bmatrix}\right)=\frac{1}{\det(\gamma)}
Q\left(\gamma\begin{bmatrix}\omega_{Q^\gamma}\\1\end{bmatrix}\right)=
\frac{(c_3\omega_{Q^\gamma}+c_4)^2}
{\det(\gamma)}
Q\left(\begin{bmatrix}\gamma(\omega_{Q^\gamma})\\1\end{bmatrix}\right)
\end{equation*}
and $\gamma(\omega^{Q^\gamma})\in\mathbb{H}$, we get
\begin{equation*}
\omega_Q=\gamma(\omega_{Q^\gamma}).
\end{equation*}
\end{remark}

\begin{definition}\label{equivalencerelation}
The equivalence relation $\sim_{\Gamma_{F,\,1}(N)}$ on $\mathcal{Q}_F(N,\,d_K)$ is given via
Lemma \ref{action} as
\begin{equation*}
Q\sim_{\Gamma_{F,\,1}(N)} Q'\quad
\Longleftrightarrow\quad Q'=Q^\gamma~\textrm{for some}~\gamma\in\Gamma_{F,\,1}(N).
\end{equation*}
We write the set of equivalence classes as
\begin{equation*}
\mathcal{C}_F(N,\,d_K)=\mathcal{Q}_F(N,\,d_K)/\sim_{\Gamma_{F,\,1}(N)}.
\end{equation*}
\end{definition}

\section {Fractional ideals associated with quadratic forms}

Let $I_K(N\mathcal{O}_K)$ be the group of fractional ideals of $K$ relatively prime to $N\mathcal{O}_K$, and
let $P_{K,\,1}(N\mathcal{O}_K)$ be its subgroup given by
\begin{equation*}
P_{K,\,1}(N\mathcal{O}_K)=\{\nu\mathcal{O}_K~|~\nu\in K^\times~\textrm{such that}~
\nu\equiv^*1\Mod{N\mathcal{O}_K}\}.
\end{equation*}
For each prime ideal $\mathfrak{p}$ of $\mathcal{O}_K$, let $n_\mathfrak{p}$
be the exponent of $\mathfrak{p}$ in the prime ideal factorization of $N\mathcal{O}_K$.
For two elements $\nu$ and $\lambda$ of $K^\times$, the multiplicative congruence
$\nu\equiv^*\lambda\Mod{N\mathcal{O}_K}$ means that
\begin{equation*}
v_\mathfrak{p}\left(\frac{\nu}{\lambda}-1\right)\geq n_\mathfrak{p}
\quad\textrm{for every prime ideal $\mathfrak{p}$ of $\mathcal{O}_K$ dividing $N\mathcal{O}_K$},
\end{equation*}
where $v_\mathfrak{p}$ is the exponential valuation corresponding to $\mathfrak{p}$
(\cite[$\S$IV.1]{Janusz}).
Let $\mathcal{C}(N\mathcal{O}_K)$ be the ray class group of $K$ modulo $N\mathcal{O}_K$, namely,
\begin{equation*}
\mathcal{C}(N\mathcal{O}_K)=I_K(N\mathcal{O}_K)/P_{K,\,1}(N\mathcal{O}_K).
\end{equation*}
The aim of this section is to assign
a ray class in $\mathcal{C}(N\mathcal{O}_K)$ to each form class in $\mathcal{C}_F(N,\,d_K)$.

\begin{lemma}\label{aw1o}
If $Q=ax^2+bxy+cy^2\in\mathcal{Q}_F$, then
\begin{equation*}
[a\omega_Q,\,1]_F=\mathcal{O}_K\quad
\Longleftrightarrow\quad
d_Q=\varepsilon^2 d_K~
\textrm{for some}~\varepsilon\in\mathcal{O}_F^\times~\textrm{such that}~\varepsilon>0.
\end{equation*}
\end{lemma}
\begin{proof}
Assume that $[a\omega_Q,\,1]_F=\mathcal{O}_K$. Then we have
$[a\omega_Q,\,1]_F=[\omega_K,\,1]_F$ by Lemma \ref{basis}, and hence
\begin{equation*}
\begin{bmatrix}a\omega_Q\\1\end{bmatrix}=
\gamma\begin{bmatrix}\omega_K\\1\end{bmatrix}
\quad\textrm{for some}~\gamma\in\mathrm{GL}_2(\mathcal{O}_F).
\end{equation*}
Here, if we let $\varepsilon=\det(\gamma)$, then we get $\varepsilon>0$ due to the fact
$a\omega_Q,\,\omega_K\in\mathbb{H}$.
Taking determinant on both sides of
\begin{equation*}
\begin{bmatrix}a\omega_Q & a\overline{\omega_Q}\\
1&1\end{bmatrix}=
\gamma\begin{bmatrix}\omega_K & \overline{\omega_K}\\
1&1\end{bmatrix}
\end{equation*}
and squaring, we obtain
$d_Q=\varepsilon^2 d_K$.
\par
Conversely, assume that
\begin{equation}\label{ded}
d_Q=\varepsilon^2 d_K\quad\textrm{for some}~\varepsilon\in\mathcal{O}_F^\times
~\textrm{such that}~\varepsilon>0.
\end{equation}
Since $a\omega_Q$ is a zero of the polynomial
\begin{equation*}
aQ(a^{-1}x,\,1)=x^2+bx+ac\in\mathcal{O}_F[x],
\end{equation*}
we achieve the inclusion
\begin{equation*}
[a\omega_Q,\,1]_F\subseteq\mathcal{O}_K.
\end{equation*}
On the other hand, we see from (\ref{ded}) that
\begin{equation*}
\omega_K=\frac{-b_K+\sqrt{d_K}}{2}=
\varepsilon^{-1}(a\omega_Q)+\frac{b\varepsilon^{-1}-b_K}{2}.
\end{equation*}
Note that $\displaystyle\frac{b\varepsilon^{-1}-b_K}{2}\in\mathcal{O}_F$
owing to the fact $\omega_K,\,\varepsilon^{-1}(a\omega_Q)\in\mathcal{O}_K$. This holds the converse inclusion
\begin{equation*}
\mathcal{O}_K=[\omega_K,\,1]_F\subseteq[a\omega_Q,\,1]_F,
\end{equation*}
and hence $[a\omega_Q,\,1]_F=\mathcal{O}_K$.
\end{proof}

\begin{lemma}\label{fractional}
Let $Q=ax^2+bxy+cy^2\in\mathcal{Q}_F(1,\,d_K)$ and $\mathfrak{a}=[\omega_Q,\,1]_F$. Then,
\begin{enumerate}
\item[\textup{(i)}] $\mathfrak{a}\in I_K(\mathcal{O}_K)$.
\item[\textup{(ii)}] $a\mathfrak{a}\overline{\mathfrak{a}}=\mathcal{O}_K$.
\item[\textup{(iii)}] $\mathfrak{a}\in I_K(N\mathcal{O}_K)$ if and only if
$Q\in\mathcal{Q}_F(N,\,d_K)$.
\end{enumerate}
\end{lemma}
\begin{proof}
Since $d_Q=d_K$, we get
$[a\omega_Q,\,1]_F=\mathcal{O}_K$ by Lemma \ref{aw1o}.
Let
\begin{equation*}
\mathfrak{b}=a\mathfrak{a}=[a\omega_Q,\,a]_F~
\end{equation*}
which is a subset of $\mathcal{O}_K$.
\begin{enumerate}
\item[(i)]
Since
\begin{eqnarray*}
&&a\omega_Q\cdot a\omega_Q=a(a\omega_Q^2)=a(-b\omega_Q-c)=-b(a\omega_Q)-ca\in\mathfrak{b}\quad
\textrm{and}\\
&&a\omega_Q\cdot a=a(a\omega_Q)\in\mathfrak{b},
\end{eqnarray*}
$\mathfrak{b}$ is an ideal of $\mathcal{O}_K=[a\omega_Q,\,1]_F$.
Thus
$\mathfrak{a}=a^{-1}\mathfrak{b}$ is a fractional ideal of $K$.
\item[(ii)]
We derive that
\begin{equation}\label{aabara}
a\mathfrak{a}\overline{\mathfrak{a}}=
a[\omega_Q,\,1]_F[\overline{\omega_Q},\,1]_F
=\mathcal{O}_F(a\omega_Q\overline{\omega_Q})+
\mathcal{O}_F(a\omega_Q)+\mathcal{O}_F(a\overline{\omega_Q})+\mathcal{O}_Fa.
\end{equation}
Since $a\omega_Q^2+b\omega_Q+c=0$, we see that
\begin{equation*}
a\omega_Q\overline{\omega_Q}(=c),\,a\omega_Q,\,a\overline{\omega_Q},\,a\in\mathcal{O}_K,
\end{equation*}
from which we obtain the inclusion $a\mathfrak{a}\overline{\mathfrak{a}}\subseteq\mathcal{O}_K$.
On the other hand, we deduce from (\ref{aabara}) that
\begin{equation*}
a\mathfrak{a}\overline{\mathfrak{a}}\supseteq
\mathcal{O}_Fc+\mathcal{O}_F(a\omega_Q+a\overline{\omega_Q})+\mathcal{O}_Fa
=\mathcal{O}_Fc+\mathcal{O}_Fb+
\mathcal{O}_Fa\ni 1
\end{equation*}
due to the fact $\gcd(a,\,b,\,c)=1$ and Assumption \ref{classnumberone}
(see (\cite[$\S$V.2]{Aluffi}).
This implies the converse inclusion $a\mathfrak{a}\overline{\mathfrak{a}}\supseteq\mathcal{O}_K$, and hence
$a\mathfrak{a}\overline{\mathfrak{a}}=\mathcal{O}_K$.
\item[(iii)] Assume that $\mathfrak{a}$ is relatively prime to $N\mathcal{O}_K$.
Since $\displaystyle\mathfrak{a}\overline{\mathfrak{a}}=\frac{1}{a}\mathcal{O}_K$ by (ii), we must have $\gcd(a,\,N)=1$.
\par
Conversely, assume that $\gcd(a,\,N)=1$. Since
\begin{equation*}
a\mathcal{O}_K=a[a\omega_Q,\,1]_F\subseteq[a\omega_Q,\,a]_F=\mathfrak{b},
\end{equation*}
$\mathfrak{b}$ divides $a\mathcal{O}_K$.
Therefore, $\mathfrak{a}=a^{-1}\mathfrak{b}$ is relatively prime to $N\mathcal{O}_K$.
\end{enumerate}
\end{proof}

\begin{lemma}\label{samezero}
If $Q,\,Q'\in\mathcal{Q}_F(N,\,d_K)$, then
\begin{equation*}
\omega_Q=\omega_{Q'}\quad\Longleftrightarrow\quad
Q=Q'.
\end{equation*}
\end{lemma}
\begin{proof}
The ``$\Longleftarrow$" part is obvious.
\par
Conversely, assume that $z=\omega_Q=\omega_{Q'}$.
If we let $Q(x,\,y)=ax^2+bxy+cy^2$ and $Q'(x,\,y)=a'x^2+b'xy+c'y^2$, then we have
\begin{equation*}
\min(z,\,F)=x^2+\frac{b}{a}x+\frac{c}{a}=
x^2+\frac{b'}{a'}x+\frac{c'}{a'},
\end{equation*}
and so
\begin{equation}\label{baca}
b=a\cdot\frac{b'}{a'}\quad\textrm{and}\quad c=a\cdot\frac{c'}{a'}.
\end{equation}
Here we recall that $\mathcal{O}_F$ is a UFD by Assumption \ref{classnumberone}.
We claim that $a'\,|\,a$. Otherwise,
$a'\nmid a$ and
there is an irreducible element $\nu$ in $\mathcal{O}_F$ such that
$\nu^n\,|\,a'$ but $\nu^n\nmid a$ for some $n\in\mathbb{Z}_{>0}$. We then
attain by (\ref{baca}) and the fact $b,\,c\in\mathcal{O}_F$ that
$\nu\,|\, b'$ and $\nu\,|\,c'$, which contradicts $\gcd(a',\,b',\,c')=1$.
In a similar way, we see $a\,|\,a'$, from which we obtain $a=\varepsilon a'$ for some
$\varepsilon\in\mathcal{O}_F^\times$ such that $\varepsilon>0$ owing to the fact
$a,\,a'>0$. It then follows from (\ref{baca}) that $b=\varepsilon b'$ and
$c=\varepsilon c'$, and hence $Q=\varepsilon Q'$. Moreover, since
\begin{equation*}
d_K=d_Q=\varepsilon^2 d_{Q'}=\varepsilon^2 d_K
\end{equation*}
and $\varepsilon>0$, we achieve $Q=Q'$, as desired.
\end{proof}

\begin{proposition}\label{qqww}
If $Q,\,Q'\in\mathcal{Q}_F(N,\,d_K)$, then we have
\begin{equation*}
[Q]=[Q']~\textrm{in}~
\mathcal{C}_F(N,\,d_K)\quad
\Longleftrightarrow\quad
[[\omega_Q,\,1]_F]=[[\omega_{Q'},\,1]_F]~\textrm{in}~\mathcal{C}(N\mathcal{O}_K).
\end{equation*}
\end{proposition}
\begin{proof}
By Lemma \ref{fractional}, we get
\begin{equation*}
[\omega_Q,\,1]_F,\,[\omega_{Q'},\,1]_F\in I_K(N\mathcal{O}_K).
\end{equation*}
Assume that $[Q]=[Q']$ in $\mathcal{C}_F(N,\,d_K)$, that is,
\begin{equation*}
Q'=Q^\gamma\quad\textrm{for some}~\gamma=\begin{bmatrix}c_1&c_2\\c_3&c_4\end{bmatrix}\in\Gamma_{F,\,1}(N).
\end{equation*}
We then derive $\omega_Q=\gamma(\omega_{Q'})$ by Remark \ref{actionzero}, and so
\begin{equation}\label{w1gw1}
[\omega_Q,\,1]_F=[\gamma(\omega_{Q'}),\,1]_F
=\frac{1}{c_3\omega_{Q'}+c_4}
[c_1\omega_{Q'}+c_2,\,c_3\omega_{Q'}+c_4]_F
=\frac{1}{c_3\omega_{Q'}+c_4}
[\omega_{Q'},\,1]_F.
\end{equation}
Let $a'$ be the coefficient of $x^2$ in $Q'$. We note by the facts $a'\omega_{Q'}\in\mathcal{O}_K$ and
\begin{equation*}
c_3\equiv0,\,
c_4\equiv\zeta\Mod{N\mathcal{O}_F}\quad\textrm{for some}~
\zeta\in\mathcal{O}_F^\times
\end{equation*}
that
\begin{equation*}
a'(c_3\omega_{Q'}+c_4)=c_3(a'\omega_{Q'})+a'c_4\equiv a'\zeta\Mod{N\mathcal{O}_K}.
\end{equation*}
Furthermore, since $\gcd(a',\,N)=1$, we obtain
\begin{equation*}
c_3\omega_{Q'}+c_4\equiv^*\zeta\Mod{N\mathcal{O}_K}.
\end{equation*}
Thus it follows that $(c_3\omega_{Q'}+c_4)\mathcal{O}_K\in P_{K,\,1}(N\mathcal{O}_K)$, and hence
\begin{equation*}
[[\omega_Q,\,1]_F]=[[\omega_{Q'},\,1]_F]~\textrm{in}~\mathcal{C}(N\mathcal{O}_K)
\end{equation*}
by (\ref{w1gw1}).
\par
Conversely, assume that $[[\omega_Q,\,1]_F]=[[\omega_{Q'},\,1]_F]$ in $\mathcal{C}(N\mathcal{O}_K)$, and so
\begin{equation*}
[\omega_{Q'},\,1]_F=\nu[\omega_Q,\,1]_F=[\nu\omega_Q,\,\nu]_F
\quad\textrm{for some}~\nu\in K^\times~\textrm{such that}~\nu\equiv^*1\Mod{N\mathcal{O}_K}.
\end{equation*}
We then deduce
\begin{equation*}
\begin{bmatrix}\nu\omega_Q\\\nu\end{bmatrix}=
\gamma\begin{bmatrix}
\omega_{Q'}\\1
\end{bmatrix}
\quad\textrm{for some}~\gamma=\begin{bmatrix}c_1&c_2\\c_3&c_4\end{bmatrix}\in\mathrm{GL}_2(\mathcal{O}_F).
\end{equation*}
Since
\begin{equation}\label{wnwngw}
\omega_Q=\frac{\nu\omega_Q}{\nu}=\frac{c_1\omega_{Q'}+c_2}{c_3\omega_{Q'}+c_4}
\end{equation}
and $\omega_Q,\,\omega_{Q'}\in
\mathbb{H}$, we get $\det(\gamma)>0$ and $\omega_Q=\gamma(\omega_{Q'})$. And, we have
\begin{equation*}
c_3\omega_{Q'}+c_4=\nu\equiv^*1\Mod{N\mathcal{O}_K}.
\end{equation*}
Letting $a'$ be the coefficient of $x^2$ in $Q'$, we claim by the fact $a'\omega_{Q'}\in\mathcal{O}_K$ that
\begin{equation}\label{cawac}
c_3(a'\omega_{Q'})+a'c_4\equiv a'\Mod{N\mathcal{O}_K}.
\end{equation}
Since $\mathcal{O}_K=[a'\omega_{Q'},\,1]_F$ by Lemma \ref{aw1o}
and $\gcd(a',\,N)=1$, the congruence (\ref{cawac}) yields
$c_3\equiv0,\,c_4\equiv1\Mod{N\mathcal{O}_F}$, which shows that $\gamma\in\Gamma_{F,\,1}(N)$.
Moreover, since
\begin{equation*}
\gamma(\omega_{Q^\gamma})=\omega_Q=\gamma(\omega_{Q'})
\end{equation*}
by Remark \ref{actionzero} and (\ref{wnwngw}), we attain $\omega_{Q^\gamma}=\omega_{Q'}$, and so
$Q'=Q^\gamma$ by Lemma \ref{samezero}.
Therefore we conclude $[Q]=[Q']$ in $\mathcal{C}_F(N,\,d_K)$.
\end{proof}

\begin{definition}
Define a map
\begin{equation*}
\begin{array}{ccccc}
\phi_{K,\,N}&:&\mathcal{C}_F(N,\,d_K)&\rightarrow&\mathcal{C}(N\mathcal{O}_K)\\
&&\mathrm{[}Q\mathrm{]}&\mapsto&[[\omega_Q,\,1]_F]
\end{array}
\end{equation*}
which is a well-defined injection by Proposition \ref{qqww}.
\end{definition}

\section {Form class groups as ray class groups}

In this section, we shall prove our first main theorem which asserts that
$\mathcal{C}_F(N,\,d_K)$ can be
regarded as a group isomorphic to the ray class group $\mathcal{C}(N\mathcal{O}_K)$.

\begin{lemma}\label{wz}
Let $z\in(K\cap\mathbb{H})\setminus F$.
If $[z,\,1]_F$ is a fractional ideal of $K$, then there is a form $Q$ in $\mathcal{Q}_F(1,\,d_K)$ with
$\omega_Q=z$. Moreover, if $[z,\,1]_F$ is relatively prime to $N\mathcal{O}_K$, then $Q$ belongs to $\mathcal{Q}_F(N,\,d_K)$.
\end{lemma}
\begin{proof}
Since $z\in(K\cap\mathbb{H})\setminus F$, $[K:F]=2$ and $\mathcal{O}_F$ is a UFD by Assumption \ref{classnumberone},
there exists a form $Q_0=ax^2+bxy+cy^2$ in $\mathcal{Q}_F$ such that
$\omega_{Q_0}=z$. Consider the set
\begin{equation*}
S=\{\nu\in K~|~\nu[\omega_{Q_0},\,1]_F\subseteq[\omega_{Q_0},\,1]_F\}.
\end{equation*}
If $\nu\in S$, then
\begin{eqnarray}
\nu\omega_{Q_0}\in[\omega_{Q_0},\,1]_F\phantom{.}\quad\textrm{and}\label{nww1}\\
\nu\in[\omega_{Q_0},\,1]_F.\phantom{\quad\textrm{and}}\label{nw1}
\end{eqnarray}
Hence we get by (\ref{nw1})
\begin{equation}\label{nuwv}
\nu=u\omega_{Q_0}+v\quad\textrm{for some}~u,\,v\in\mathcal{O}_F,
\end{equation}
and so we derive that
\begin{eqnarray*}
\nu\omega_{Q_0}&=&(u\omega_{Q_0}+v)\omega_{Q_0}\\
&=&u\left(-\frac{b}{a}\omega_{Q_0}-\frac{c}{a}\right)+v\omega_{Q_0}\quad\textrm{because}~
a\omega_{Q_0}^2+b\omega_{Q_0}+c=0\\
&=&\left(-\frac{ub}{a}+v\right)\omega_{Q_0}-\frac{uc}{a}.
\end{eqnarray*}
This holds by (\ref{nww1}) that
\begin{equation*}
\frac{ub}{a},\,\frac{uc}{a}\in\mathcal{O}_F,
\end{equation*}
so we have $a\,|\,u$ by the fact $\gcd(a,\,b,\,c)=1$.
It then follows
from (\ref{nuwv})
\begin{equation*}
\nu=u\omega_{Q_0}+v\in[a\omega_{Q_0},\,1]_F,
\end{equation*}
from which we obtain the inclusion
\begin{equation}\label{saw1}
S\subseteq[a\omega_{Q_0},\,1]_F.
\end{equation}
\par
On the other hand, since we are assuming that $[z,\,1]_F=
[\omega_{Q_0},\,1]_F$ is a fractional ideal of $K$, we attain
$S=\mathcal{O}_K$. Furthermore, since $a\omega_{Q_0}$ is a zero of the polynomial
$aQ_0(a^{-1}x,\,1)=x^2+bx+ac$ ($\in\mathcal{O}_F[x]$), we get by (\ref{saw1}) that
\begin{equation*}
\mathcal{O}_K=S\subseteq[a\omega_{Q_0},\,1]_F\subseteq\mathcal{O}_K.
\end{equation*}
This yields $[a\omega_{Q_0},\,1]_F=\mathcal{O}_K$, and hence
\begin{equation*}
d_{Q_0}=\varepsilon^2d_K\quad\textrm{for some}~\varepsilon\in\mathcal{O}_F^\times~
\textrm{such that}~\varepsilon>0
\end{equation*}
by Lemma \ref{aw1o}. Observe that the form $Q=\displaystyle\frac{1}{\varepsilon}Q_0$
belongs to $\mathcal{Q}_F(1,\,d_K)$ and satisfies
\begin{equation*}
\omega_Q=\omega_{Q_0}=z.
\end{equation*}
\par
The second assertion follows immediately from Lemma \ref{fractional} (iii).
\end{proof}

\begin{lemma}[Strong approximation]\label{strong}
Let $E$ be a number field with ring of integers $\mathcal{O}_E$ and
$\mathfrak{m}$ be a nontrivial ideal of $\mathcal{O}_E$. Then, the
reduction $\mathrm{SL}_2(\mathcal{O}_E)\rightarrow\mathrm{SL}_2(\mathcal{O}_E/\mathfrak{m})$ is surjective.
\end{lemma}
\begin{proof}
See \cite[pp. 244--268]{Hurwitz}.
\end{proof}

\begin{proposition}\label{surjective}
The map $\phi_{K,\,N}$ is surjective.
\end{proposition}
\begin{proof}
Let $C\in\mathcal{C}(N\mathcal{O}_K)$. One can always take
an integral ideal $\mathfrak{a}$ in the class $C^{-1}$ (\cite[Lemma 2.3 in
Chapter IV]{Janusz}), and also take $\xi_1,\,\xi_2\in K^\times$ such that
\begin{equation}\label{axx}
\mathfrak{a}^{-1}=[\xi_1,\,\xi_2]_F\quad\textrm{and}\quad
\xi=\frac{\xi_1}{\xi_2}\in\mathbb{H}
\end{equation}
because $\mathcal{O}_K$ is a free $\mathcal{O}_F$-module of rank $2$ and $\mathcal{O}_F$ is a PID by Assumption \ref{classnumberone}
(\cite[Theorem 6.1 in Chapter IV]{Hungerford}). Since $1\in\mathfrak{a}^{-1}$, one can express $1$ as
\begin{equation}\label{1}
1=u\xi_1+v\xi_2\quad\textrm{for some}~u,\,v\in\mathcal{O}_F.
\end{equation}
Now we claim that
\begin{equation}\label{gcd}
\gcd(N,\,u,\,v)=1.
\end{equation}
Indeed, if
there is an irreducible $d$ in $\mathcal{O}_F$ dividing $\gcd(N,\,u,\,v)$, then
the fractional ideal $d\mathfrak{a}^{-1}=[d\xi_1,\,d\xi_2]_F$ of $K$ contains $1$ by (\ref{1}).
So we derive $d\mathfrak{a}^{-1}\supseteq\mathcal{O}_K$, from which $d\mathcal{O}_K\supseteq\mathfrak{a}$
and $d\mathcal{O}_K\,|\,\mathfrak{a}$.
But this contradicts the fact that $\mathfrak{a}$ is relatively prime to $N\mathcal{O}_K$.
\par
Let $e=\gcd(u,\,v)$ up to associate in $\mathcal{O}_F$. Since $\mathcal{O}_F$ is a PID, we may express $e$ as
\begin{equation}\label{euuvv}
e=uu'+vv'\quad\textrm{for some}~u',\,v'\in\mathcal{O}_F
\end{equation}
(\cite[$\S$V.2]{Aluffi}). Furthermore, since $\gcd(N,\,e)=1$ by (\ref{gcd}),
we write $1$ as
\begin{equation}\label{1nnee}
1=NN'+ee'\quad\textrm{for some}~N',\,e'\in\mathcal{O}_F.
\end{equation}
We then obtain by (\ref{euuvv}) and (\ref{1nnee}) that
\begin{equation*}
1\equiv ee'
\equiv uu'e'+vv'e'
\Mod{N\mathcal{O}_F},
\end{equation*}
and so
\begin{equation*}
\begin{bmatrix}v'e' & -u'e'\\
u & v\end{bmatrix}\in\mathrm{SL}_2(\mathcal{O}_F/N\mathcal{O}_F).
\end{equation*}
Hence there exists a matrix
\begin{equation*}
\gamma=\begin{bmatrix}\mathrm{*}&\mathrm{*}\\
\widetilde{u}&\widetilde{v}\end{bmatrix}\in\mathrm{SL}_2(\mathcal{O}_F)
\quad\textrm{such that}~\widetilde{u}\equiv u,\,\widetilde{v}\equiv v\Mod{N\mathcal{O}_F}
\end{equation*}
by Lemma \ref{strong}.
\par
Let $z=\gamma(\xi)$ ($\in K\cap\mathbb{H}$). Then we deduce that
\begin{equation}\label{w1gx1}
[z,\,1]_F=[\gamma(\xi),\,1]_F
=\frac{1}{\widetilde{u}\xi+\widetilde{v}}[\xi,\,1]_F=
\frac{1}{\widetilde{u}\xi_1+\widetilde{v}\xi_2}[\xi_1,\,\xi_2]_F=
\frac{1}{\widetilde{u}\xi_1+\widetilde{v}\xi_2}\mathfrak{a}^{-1}.
\end{equation}
Here we observe by (\ref{axx}) and (\ref{1}) that
\begin{equation*}
\widetilde{u}\xi_1+\widetilde{v}\xi_2-1=
\widetilde{u}\xi_1+\widetilde{v}\xi_2-(u\xi_1+v\xi_2)=
(\widetilde{u}-u)\xi_1+(\widetilde{v}-v)\xi_2\in N\mathfrak{a}^{-1},
\end{equation*}
which show that
\begin{equation*}
\widetilde{u}\xi_1+\widetilde{v}\xi_2\equiv^*1\Mod{N\mathcal{O}_K}
\end{equation*}
since $\mathfrak{a}$ is relatively prime to $N\mathcal{O}_K$. Thus we obtain by (\ref{w1gx1}) that
\begin{equation*}
[[z,\,1]_F]=[\mathfrak{a}^{-1}]=C\quad\textrm{in}~\mathcal{C}(N\mathcal{O}_K).
\end{equation*}
Moreover, there is a form $Q$ in $\mathcal{Q}_F(N,\,d_K)$ with
$\omega_Q=z$ by Lemma \ref{wz}. Therefore we achieve
\begin{equation*}
\phi_{K,\,N}([Q])=C,
\end{equation*}
which proves that $\phi_{K,\,N}$ is surjective.
\end{proof}

\begin{theorem}\label{main1}
We may consider $\mathcal{C}_F(N,\,d_K)$ as a group isomorphic to
$\mathcal{C}(N\mathcal{O}_K)$ via the bijection
$[Q]\mapsto[[\omega_Q,\,1]_F]$.
\end{theorem}
\begin{proof}
By Propositions \ref{qqww} and \ref{surjective}, the
map $\phi_{K,\,N}:\mathcal{C}_F(N,\,d_K)\rightarrow\mathcal{C}(N\mathcal{O}_K)$ is bijective.
One can regard $\mathcal{C}_F(N,\,d_K)$ as a group isomorphic to $\mathcal{C}(N\mathcal{O}_K)$
by endowing it with a group structure so that
$\phi_{K,\,N}$ becomes an isomorphism.
\end{proof}

\section {Totally positivity}

Under the assumption that $F$ has narrow class number one, we shall
show that each form class $[Q]$
in $\mathcal{C}_F(N,\,d_K)$ has a ``totally positive" representative,
which is essential to introduce a CM-point associated with $[Q]$.

\begin{lemma}\label{qqcwgw}
If $Q,\,Q'\in\mathcal{Q}_F(N,\,d_K)$, then
\begin{equation*}
[Q]=[Q']~\textrm{in}~\mathcal{C}_F(N,\,d_K)\quad
\Longleftrightarrow\quad-\overline{\omega_Q}=\gamma(-\overline{\omega_{Q'}})~
\textrm{for some}~\gamma\in\Gamma_{F,\,1}(N).
\end{equation*}
\end{lemma}
\begin{proof}
Assume that $[Q]=[Q']$ in $\mathcal{C}_F(N,\,d_K)$, namely,
\begin{equation*}
Q'=Q^\alpha\quad\textrm{for some}~\alpha=\begin{bmatrix}a_1&a_2\\
a_3&a_4\end{bmatrix}\in\Gamma_{F,\,1}(N).
\end{equation*}
Then we see by Remark \ref{actionzero} that
\begin{equation*}
\omega_Q=\alpha(\omega_{Q'})=\frac{a_1\omega_{Q'}+a_2}{a_3\omega_{Q'}+a_4},
\end{equation*}
and so
\begin{equation*}
-\overline{\omega_Q}=\frac{a_1(-\overline{\omega_{Q'}})-a_2}{-a_3(-\overline{\omega_{Q'}})+a_4}
=\gamma(-\overline{\omega_{Q'}})\quad\textrm{with}~
\gamma=\begin{bmatrix}\phantom{-}a_1&-a_2\\-a_3&\phantom{-}a_4\end{bmatrix}\in\Gamma_{F,\,1}(N).
\end{equation*}
\par
Conversely, assume that
\begin{equation*}
-\overline{\omega_Q}=\gamma(-\overline{\omega_{Q'}})\quad
\textrm{for some}~\gamma=\begin{bmatrix}c_1&c_2\\c_3&c_4\end{bmatrix}\in\Gamma_{F,\,1}(N).
\end{equation*}
Letting $\alpha=\begin{bmatrix}\phantom{-}c_1&-c_2\\
-c_3&\phantom{-}c_4\end{bmatrix}$ ($\in\Gamma_{F,\,1}(N)$) we derive that
\begin{eqnarray*}
\alpha(\omega_{Q^\alpha})&=&
\omega_Q\quad\textrm{by Remark \ref{actionzero}}\\
&=&\alpha(\omega_{Q'})\quad\textrm{as above},
\end{eqnarray*}
and so $\omega_{Q^\alpha}=\omega_{Q'}$.
Thus we attain by Lemma \ref{samezero} that $Q'=Q^\alpha$, and hence
$[Q]=[Q']$ in $\mathcal{C}_F(N,\,d_K)$.
\end{proof}

For $\nu\in F$, we write $\nu\gg0$ when it is totally positive.
Let
\begin{eqnarray*}
\mathcal{Q}_F^+(N,\,d_K)&=&\{ax^2+bxy+cy^2\in\mathcal{Q}_F(N,\,d_K)~|~a\gg0\}\quad\textrm{and}\\
\Gamma_{F,\,1}^+(N)&=&\{\gamma\in\Gamma_{F,\,1}(N)~|~\det(\gamma)\gg0\}.
\end{eqnarray*}
Let $[K:\mathbb{Q}]=2g$ and $(K,\,\{\varphi_i\}_{i=1}^g)$
be a CM-type. That is, $\{\varphi_1,\,\varphi_2,
\ldots,\,\varphi_g,\,
\varphi_1\rho,\,\varphi_2\rho,
\ldots,\,\varphi_g\rho\}$ is  exactly the set of all embeddings of $K$ into $\mathbb{C}$,
where $\rho$ stands for the complex conjugation.
Throughout this section, we fix a CM-type $(K,\,\{\varphi_i\}_{i=1}^g)$ in such a way that
\begin{equation}\label{rirh}
\varphi_1=\mathrm{id}_K\quad\textrm{and}\quad \varphi_i(\omega_K)\in\mathbb{H}\quad (i=1,\,2,\,\ldots,\,g).
\end{equation}
Note that
\begin{equation}\label{rbar}
\varphi_i(\overline{z})=\overline{\varphi_i(z)}\quad\textrm{for all}~z\in K~\textrm{and}~
i=1,\,2,\,\ldots,\,g
\end{equation}
(\cite[Lemma 18.2 (i)]{Shimura}).

\begin{lemma}\label{+QQC}
If $Q,\,Q'\in\mathcal{Q}_F^+(N,\,d_K)$, then
\begin{equation*}
[Q]=[Q']~\textrm{in}~\mathcal{C}_N(N,\,d_K)\quad
\Longleftrightarrow\quad
-\overline{\omega_Q}=\gamma(-\overline{\omega_{Q'}})~\textrm{for some}~\gamma\in
\Gamma_{F,\,1}^+(N).
\end{equation*}
\end{lemma}
\begin{proof}
The ``$\Longleftarrow$\," part is clear due to Lemma \ref{qqcwgw}.
\par
Conversely, assume that $[Q]=[Q']$ in $\mathcal{C}_F(N,\,d_K)$. Then we have
\begin{equation*}
-\overline{\omega_Q}=\gamma(-\overline{\omega_{Q'}})\quad
\textrm{for some}~\gamma=\begin{bmatrix}c_1&c_2\\c_3&c_4\end{bmatrix}\in\Gamma_{F,\,1}(N)
\end{equation*}
by Lemma \ref{qqcwgw}. Let $Q=ax^2+bxy+cy^2$. Since $\varphi_i(a)>0$, $\varphi_i(w_K)\in\mathbb{H}$ ($i=1,\,2,\,\ldots,\,g$)
and $F$ is totally real, we derive that
\begin{equation}\label{rgw}
\varphi_i(\gamma(-\overline{\omega_{Q'}}))=\varphi_i(-\overline{\omega_Q})=\varphi_i\left(
\frac{b+\sqrt{d_K}}{2a}\right)=
\frac{\varphi_i(\omega_K)}{\varphi_i(a)}+
\varphi_i\left(\frac{b_K+b}{2a}\right)\in\mathbb{H}.
\end{equation}
It then follows that
\begin{equation}\label{girw}
\gamma^{(i)}(\varphi_i(-\overline{\omega_{Q'}}))=
\varphi_i(\gamma(-\overline{\omega_{Q'}}))\in\mathbb{H}\quad
\textrm{where}~\gamma^{(i)}=\begin{bmatrix}
\varphi_i(c_1)&\varphi_i(c_2)\\
\varphi_i(c_3)&\varphi_i(c_4)\end{bmatrix}.
\end{equation}
In a similar way to (\ref{rgw}), one can also show that
$\varphi_i(-\overline{\omega_{Q'}})\in\mathbb{H}$. Thus we claim
by (\ref{girw}) that
\begin{equation*}
\det(\gamma^{(i)})>0\quad(i=1,\,2,\,\ldots,\,g)
\end{equation*}
(\cite[Lemma 1.1]{Silverman}), and hence $\gamma\in\Gamma_{F,\,1}^+(N)$.
\end{proof}

\begin{assumption}\label{narrowclassnumberone}
From now one, we assume that the narrow class number $h_F^+$ of $F$ is one.
\end{assumption}

\begin{lemma}\label{+representative}
We have
\begin{equation*}
\mathcal{C}_F(N,\,d_K)=\{[Q]~|~Q\in\mathcal{Q}_F^+(N,\,d_K)\}.
\end{equation*}
\end{lemma}
\begin{proof}
Let $Q=ax^2+bxy+cy^2\in\mathcal{Q}_F(N,\,d_K)$. Since we are assuming $h_F^+=1$, we get
\begin{equation*}
a\mathcal{O}_F=a'\mathcal{O}_F\quad\textrm{for some}~a'\in F~\textrm{with}~a'\gg0,
\end{equation*}
and so
\begin{equation*}
a'=\zeta a\quad\textrm{for some}~\zeta\in\mathcal{O}_F^\times~\textrm{such that}~\zeta>0.
\end{equation*}
If we let
\begin{equation*}
Q'=Q^\gamma
\quad\textrm{with}~\gamma=\begin{bmatrix}\zeta&0\\0&1\end{bmatrix}\in\Gamma_{F,\,1}(N),
\end{equation*}
then we conclude that
\begin{equation*}
[Q]=[Q']\quad\textrm{and}\quad
Q'=a'x^2+bxy+\zeta^{-1}cy^2\in\mathcal{Q}_F^+(N,\,d_K).
\end{equation*}
This completes the proof.
\end{proof}

\section {Ray class invariants}\label{defininginvariant}

Throughout this section, we let $f=f(z_1,\,z_2,\,\ldots,\,z_g)$ be a Hilbert modular function for
$\Gamma_{F,\,1}^+(N)$ with rational Fourier coefficients, and let
$C\in\mathcal{C}(N\mathcal{O}_K)$.
We shall define a ray class invariant $f(C)$ of $C$ as
a singular value of $f$.
\par
First, take an integral ideal $\mathfrak{c}$ in the ray class $C$ (\cite[Lemma 2.3 in Chapter IV]{Janusz}).

\begin{lemma}\label{x1x2}
There are elements $\xi_1$ and $\xi_2$ of $K$ such that
$\mathfrak{c}^{-1}=[\xi_1,\,\xi_2]_F$ and
\begin{equation}\label{rhx}
\varphi_i(\xi)\in\mathbb{H}\quad\textrm{with}~\xi=\frac{\xi_1}{\xi_2}
\quad(i=1,\,2,\,\ldots,\,g).
\end{equation}
\end{lemma}
\begin{proof}
By Theorem \ref{main1} and Lemma \ref{+representative}, we have
\begin{equation*}
C^{-1}=[[\omega_Q,\,1]_F]\quad
\textrm{for some}~Q=ax^2+bxy+cy^2\in\mathcal{Q}_F^+(N,\,d_K).
\end{equation*}
So, we get
\begin{equation*}
\mathfrak{c}^{-1}=\nu[\omega_Q,\,1]_F
\quad\textrm{for some}~\nu\in K^\times~\textrm{such that}~
\nu\equiv^*1\Mod{N\mathcal{O}_K}.
\end{equation*}
If we take
\begin{equation*}
\xi_1=\nu\omega_Q\quad\textrm{and}\quad \xi_2=\nu,
\end{equation*}
then we see that $\mathfrak{c}^{-1}=[\xi_1,\,\xi_2]_F$ and
\begin{equation}\label{rawbh}
\varphi_i\left(\frac{\xi_1}{\xi_2}\right)=\varphi_i(\omega_Q)=
\varphi_i\left(
\frac{1}{a}\left(\omega_K+\frac{b_K-b}{2}\right)
\right)\in\mathbb{H}\quad(i=1,\,2,\,\ldots,\,g)
\end{equation}
because $a\gg0$ and $\varphi_i(\omega_K)\in\mathbb{H}$.
\end{proof}

Now, let $\xi_1,\,\xi_2,\,\xi$ be as in Lemma \ref{x1x2}.
Since
\begin{equation*}
\mathcal{O}_K=[\omega_K,\,1]_F~\subseteq~\mathfrak{c}^{-1}=[\xi_1,\,\xi_2]_F,
\end{equation*}
we deduce
\begin{equation}\label{wAx}
\begin{bmatrix}
\omega_K\\1
\end{bmatrix}=A\begin{bmatrix}
\xi_1\\\xi_2
\end{bmatrix}\quad\textrm{for some}~A\in M_2(\mathcal{O}_F).
\end{equation}

\begin{lemma}\label{det>>0}
The matrix $A$ satisfies the following properties.
\begin{enumerate}
\item[\textup{(i)}] $\det(A)\gg0$.
\item[\textup{(ii)}] $\gcd(\det(A),\,N)=1$.
\end{enumerate}
\end{lemma}
\begin{proof}
\begin{enumerate}
\item[(i)] Let $A=\begin{bmatrix}a_1&a_2\\a_3&a_4\end{bmatrix}$.
We attain by (\ref{wAx}) that
\begin{equation}\label{II}
\mathrm{Im}(\varphi_i(\omega_K))=
\mathrm{Im}\left(\varphi_i\left(\frac{a_1\xi+a_2}{a_3\xi+a_4}\right)\right)
=\frac{\varphi_i(\det(A))\mathrm{Im}(\varphi_i(\xi))}
{|\varphi_i(a_3)\varphi_i(\xi)+\varphi_i(a_4)|^2}\quad(i=1,\,2,\,\ldots,\,g)
\end{equation}
(\cite[Lemma 1.1 in Chapter I]{Silverman}). Since $\varphi_i(\omega_K)$ and $\varphi_i(\xi)$
belong to $\mathbb{H}$ by (\ref{rirh}) and (\ref{rhx}),
we achieve from (\ref{II}) that $\varphi_i(\det(A))>0$, and hence $\det(A)\gg0$.
\item[(ii)] Again by (\ref{wAx}), we have
\begin{equation*}
\begin{bmatrix}
\omega_K & \overline{\omega_K}\\
1&1\end{bmatrix}
=A\begin{bmatrix}\xi_1&\overline{\xi_1}\\
\xi_2&\overline{\xi_2}\end{bmatrix}.
\end{equation*}
We then derive that
\begin{equation*}
d_K\mathcal{O}_F=\det(A)^2N_{K/F}(\mathfrak{c}^{-1})^2d_K\mathcal{O}_F
\end{equation*}
(\cite[$\S$I.7]{Janusz} and \cite[$\S$III.3]{Lang}), and so
\begin{equation*}
\det(A)\mathcal{O}_F=N_{K/F}(\mathfrak{c}).
\end{equation*}
Since $\mathfrak{c}$ is relatively prime to $N\mathcal{O}_K$, we should have
$\gcd(\det(A),\,N)=1$.
\end{enumerate}
\end{proof}

Letting $d=\det(A)$, we obtain $d\gg0$ and $\gcd(d,\,N)=1$
by Lemma \ref{det>>0}. Moreover, since the reduction $\mathrm{SL}_2(\mathcal{O}_F)\rightarrow
\mathrm{SL}_2(\mathcal{O}_F/N\mathcal{O}_F)$ is surjective by Lemma \ref{strong}, we get
\begin{equation}\label{AA_1}
A\equiv\begin{bmatrix}1&0\\0&d\end{bmatrix}A_1
\Mod{M_2(N\mathcal{O}_F)}\quad
\textrm{for some}~A_1\in\mathrm{SL}_2(\mathcal{O}_F).
\end{equation}

\begin{definition}\label{rayclassinvariant}
We define
\begin{equation*}
f(C)=(f\circ A_1)(\varphi_1(\xi),\,\varphi_2(\xi),\,\ldots,\,\varphi_g(\xi)).
\end{equation*}
\end{definition}

\begin{proposition}
The value $f(C)$ is a ray class invariant of $C$. That is, $f(C)$ depends only on $C$, not on the choice
of $\mathfrak{c},\,\xi_1,\,\xi_2,\,A_1$.
\end{proposition}
\begin{proof}
Let $\mathfrak{c}'$ be an integral ideal in $C$, and let $\xi_1'$ and $\xi_2'$ be
elements of $K$ such that $\mathfrak{c}'^{-1}=[\xi_1',\,\xi_2']_F$ and
\begin{equation}\label{rhx'}
\varphi_i(\xi')\in\mathbb{H}\quad\textrm{with}~\xi'=\frac{\xi_1'}{\xi_2'}\quad
(i=1,\,2,\,\ldots,\,g).
\end{equation}
Furthermore, let $A'$ be a matrix in $M_2(\mathcal{O}_F)$ such that
\begin{equation}\label{wA'x}
\begin{bmatrix}\omega_K\\1\end{bmatrix}=
A'\begin{bmatrix}\xi_1'\\\xi_2'\end{bmatrix}.
\end{equation}
Note by Lemma \ref{det>>0} that $\det(A')\gg0$ and $\gcd(\det(A'),\,N)=1$.
Since $C=[\mathfrak{c}]=[\mathfrak{c}']$, we have
\begin{equation*}
\mathfrak{c}'=\nu\mathfrak{c}\quad
\textrm{for some}~\nu\in K^\times~\textrm{such that}~\nu
\equiv^*1\Mod{N\mathcal{O}_K},
\end{equation*}
and hence
\begin{equation*}
[\xi_1',\,\xi_2']_F=\mathfrak{c}'^{-1}=\nu^{-1}\mathfrak{c}^{-1}=
\nu^{-1}[\xi_1,\,\xi_2]_F=
[\nu^{-1}\xi_1,\,\nu^{-1}\xi_2]_F.
\end{equation*}
Therefore we establish
\begin{equation}\label{x'Blx}
\begin{bmatrix}
\xi_1'\\\xi_2'
\end{bmatrix}=B\begin{bmatrix}\nu^{-1}\xi_1\\
\nu^{-1}\xi_2\end{bmatrix}\quad\textrm{for some}~B=\begin{bmatrix}b_1&b_2\\
b_3&b_4\end{bmatrix}\in\mathrm{GL}_2(\mathcal{O}_F),
\end{equation}
from which we obtain
\begin{equation}\label{xBx'}
\xi=\frac{\nu^{-1}\xi_1}{\nu^{-1}\xi_2}=B^{-1}\left(\frac{\xi_1'}{\xi_2'}\right)=B^{-1}(\xi').
\end{equation}
Here, one can show by (\ref{rhx}) and (\ref{rhx'})
that $\det(B)\gg0$ in a similar way to Lemma \ref{det>>0} (i)
\par
On the other hand, since $\mathfrak{c}$ and $\mathfrak{c}'=\nu\mathfrak{c}$ are ideals of $\mathcal{O}_K$, so is
$(\nu-1)\mathfrak{c}$. Moreover, since
$\nu\equiv^*1\Mod{N\mathcal{O}_K}$ and $\mathfrak{c}$ is relatively prime to $N\mathcal{O}_K$, we attain
$(\nu-1)\mathfrak{c}\subseteq N\mathcal{O}_K$ and so
\begin{equation*}
[(\nu-1)\omega_K,\,
\nu-1]_F=(\nu-1)\mathcal{O}_K\subseteq N\mathfrak{c}^{-1}=[N\xi_1,\,N\xi_2]_F.
\end{equation*}
This yields
\begin{equation*}
\begin{bmatrix}(\nu-1)\omega_K\\
\nu-1\end{bmatrix}=
A''\begin{bmatrix}N\xi_1\\N\xi_2\end{bmatrix}\quad
\textrm{for some}~A''\in M_2(\mathcal{O}_F).
\end{equation*}
We then see that
\begin{eqnarray*}
NA''\begin{bmatrix}\xi_1\\
\xi_2\end{bmatrix}&=&\nu\begin{bmatrix}\omega_K\\1\end{bmatrix}-\begin{bmatrix}\omega_K
\\1\end{bmatrix}\\
&=&\nu A'\begin{bmatrix}\xi_1'\\\xi_2'\end{bmatrix}-A\begin{bmatrix}\xi_1\\\xi_2\end{bmatrix}
\quad\textrm{by (\ref{wAx}) and (\ref{wA'x})}\\
&=&A'B\begin{bmatrix}\xi_1\\\xi_2\end{bmatrix}-A\begin{bmatrix}\xi_1\\\xi_2\end{bmatrix}
\quad\textrm{by (\ref{x'Blx})}\\
&=&(A'B-A)\begin{bmatrix}\xi_1\\\xi_2\end{bmatrix},
\end{eqnarray*}
and hence $NA''=A'B-A$. And, we know that
\begin{eqnarray*}
A'&=&NA''B^{-1}+AB^{-1}\quad
\textrm{since $B$ in (\ref{x'Blx}) belongs to $\mathrm{GL}_2(\mathcal{O}_F)$}\\
&\equiv&AB^{-1}\Mod{M_2(N\mathcal{O}_F)}\\
&\equiv&\begin{bmatrix}1&0\\0&d\end{bmatrix}A_1B^{-1}
\Mod{M_2(N\mathcal{O}_F)}\quad\textrm{by (\ref{AA_1})}\\
&\equiv&\begin{bmatrix}
1&0\\0&db^{-1}\end{bmatrix}
\left(
\begin{bmatrix}1&0\\0&b\end{bmatrix}A_1B^{-1}
\right)\Mod{M_2(N\mathcal{O}_F)}\quad
\textrm{with}~b=\det(B)~(\in\mathcal{O}_F^\times).
\end{eqnarray*}
Here, we observe that
\begin{equation*}
\begin{bmatrix}1&0\\0&b\end{bmatrix}A_1B^{-1}\in\mathrm{SL}_2(\mathcal{O}_F).
\end{equation*}
Thus, if we let
$d'=\det(A')$ and $A'_1$ be a matrix in $\mathrm{SL}_2(\mathcal{O}_F)$ such that
\begin{equation*}
A'\equiv\begin{bmatrix}1&0\\0&d'\end{bmatrix}A_1'\Mod{M_2(N\mathcal{O}_F)},
\end{equation*}
then we get
$d'\equiv db^{-1}\Mod{N\mathcal{O}_F}$
and
\begin{equation}\label{A'1AB}
A_1'\equiv\begin{bmatrix}1&0\\0&b\end{bmatrix}A_1B^{-1}\Mod{M_2(N\mathcal{O}_F)}.
\end{equation}
\par
Now, we derive that
\begin{eqnarray*}
&&(f\circ A_1')(\varphi_1(\xi'),\,\varphi_2(\xi'),\,\ldots,\,\varphi_g(\xi'))\\
&=&\left(f\circ\begin{bmatrix}1&0\\0&b\end{bmatrix}
A_1B^{-1}\right)(\varphi_1(\xi'),\,\varphi_2(\xi'),\,\ldots,\,\varphi_g(\xi'))\\
&&\hspace{5cm}\textrm{by (\ref{A'1AB}) and the fact that $f$ is modular for $\Gamma_{F,\,1}^+(N)$}\\
&=&((f\circ A_1)\circ B^{-1})(\varphi_1(\xi'),\,\varphi_2(\xi'),\,\ldots,\,\varphi_g(\xi'))\quad
\textrm{because}~\begin{bmatrix}1&0\\0&b\end{bmatrix}\in\Gamma_{F,\,1}^+(N)\\
&=&(f\circ A_1)((B^{-1})^{(1)}(\varphi_1(\xi')),\,
(B^{-1})^{(2)}(\varphi_2(\xi')),\,\ldots,\,(B^{-1})^{(g)}(\varphi_g(\xi')))\\
&&\hspace{5cm}\textrm{where}~(B^{-1})^{(i)}=\begin{bmatrix}\varphi_i(b_1)&\varphi_i(b_2)\\
\varphi_i(b_3)&\varphi_i(b_4)\end{bmatrix}\quad(i=1,\,2,\,\ldots,\,g)\\
&=&(f\circ A_1)(\varphi_1(B^{-1}(\xi')),\,\varphi_2(B^{-1}(\xi')),\,
\ldots,\,\varphi_g(B^{-1}(\xi')))\\
&=&(f\circ A_1)(\varphi_1(\xi),\,\varphi_2(\xi),\,\ldots,\,\varphi_g(\xi))\quad
\textrm{by (\ref{xBx'})}.
\end{eqnarray*}
This proves that $f(C)$ is independent of the choice of
$\mathfrak{c},\,\xi_1,\,\xi_2,\,A_1$.
\end{proof}

\begin{lemma}\label{f(w)}
If we let
$C=\phi_{K,\,N}([Q])$ for some $Q\in\mathcal{Q}_F^+(N,\,d_K)$ by \textup{Theorem \ref{main1}}
and \textup{Lemma \ref{+representative}}, then we have
\begin{equation*}
f(C)=f(\varphi_1(-\overline{\omega_Q}),\,
\varphi_2(-\overline{\omega_Q}),\,\ldots,\,
\varphi_g(-\overline{\omega_Q})).
\end{equation*}
\end{lemma}
\begin{proof}
Let $Q=ax^2+bxy+cy^2$. If we set $e=|(\mathcal{O}_F/N\mathcal{O}_F)^\times|$, then we have
\begin{equation}\label{ae1}
a^e\equiv1\Mod{N\mathcal{O}_F}.
\end{equation}
And, we may take
$\mathfrak{c}=a^e[\omega_Q,\,1]_F$ as an integral ideal of $K$ belonging to the ray class $C$. Since
\begin{equation*}
\mathfrak{c}\overline{\mathfrak{c}}=a^{2e}\frac{1}{a}\mathcal{O}_F=a^{2e-1}\mathcal{O}_F
\end{equation*}
by Lemma \ref{fractional} (ii), we achieve
\begin{equation*}
\mathfrak{c}^{-1}=\frac{1}{a^{2e-1}}\overline{\mathfrak{c}}
=\frac{1}{a^{e-1}}[-\overline{\omega_Q},\,1]_F.
\end{equation*}
If we take
\begin{equation*}
\xi_1=-\frac{1}{a^{e-1}}\overline{\omega_Q}\quad\textrm{and}\quad
\xi_2=\frac{1}{a^{e-1}},
\end{equation*}
then we see that
\begin{equation*}
\mathfrak{c}^{-1}=[\xi_1,\,\xi_2]_F\quad\textrm{and}\quad
\xi=\frac{\xi_1}{\xi_2}=-\overline{\omega_Q}
\end{equation*}
and
\begin{equation*}
\varphi_i(\xi)=\varphi_i(-\overline{\omega_Q})=-\overline{\varphi_i(\omega_Q)}\in\mathbb{H}
\quad(i=1,\,2,\,
\ldots,\,g)
\end{equation*}
by (\ref{rbar}) and (\ref{rawbh}). We also deduce that
\begin{equation*}
\begin{bmatrix}
\omega_K\\1
\end{bmatrix}
=\begin{bmatrix}
a^e & -a^{e-1}(b+b_K)/2\\
0&a^{e-1}
\end{bmatrix}
\begin{bmatrix}
\xi_1\\\xi_2
\end{bmatrix}
\end{equation*}
and
\begin{equation*}
\begin{bmatrix}
a^e & -a^{e-1}(b+b_K)/2\\
0&a^{e-1}
\end{bmatrix}\equiv
\begin{bmatrix}1&0\\0&a^{2e-1}\end{bmatrix}
\begin{bmatrix}
1&-a^{e-1}(b+b_K)/2\\
0&1
\end{bmatrix}
\Mod{M_2(N\mathcal{O}_K)}
\end{equation*}
by (\ref{ae1}).
Therefore we obtain by Definition \ref{rayclassinvariant} that
\begin{eqnarray*}
f(C)&=&\left(f\circ\begin{bmatrix}
1&-a^{e-1}(b+b_K)/2\\
0&1
\end{bmatrix}\right)(\varphi_1(-\overline{\omega_Q}),\,
\varphi_2(-\overline{\omega_Q}),\,\ldots,\,
\varphi_g(-\overline{\omega_Q}))\\
&=&f(\varphi_1(-\overline{\omega_Q}),\,
\varphi_2(-\overline{\omega_Q}),\,\ldots,\,
\varphi_g(-\overline{\omega_Q}))\quad\textrm{because $f$ is modular for $\Gamma_{F,\,1}^+(N)$}.
\end{eqnarray*}
\end{proof}

\section {Functions on the canonical model}

We shall devote the following two sections to review
some necessary consequences in the theory of complex multiplication of abelian varieties
due to Shimura.
\par
Let $\mathbf{a}$ be the set of archimedean primes of $F$
and $\mathbf{h}$ be the set of nonarchimedean primes of $F$.
We denote by
\begin{eqnarray*}
G&=&\mathrm{GL}_2(F),\\
\displaystyle G_\mathbb{A}&=&\mathrm{GL}_2(F_\mathbb{A})~\simeq~
\left\{(\alpha_v)_v\in\prod_{v\in\mathbf{a}\cup\mathbf{h}}\mathrm{GL}_2(F_v)~|~
\alpha_v\in\mathrm{GL}_2(\mathcal{O}_{F,\,v})~\textrm{for almost all}~v\in\mathbf{h}\right\},\\
G_\mathbf{a}&=&\left\{
(\alpha_v)_v\in\mathrm{GL}_2(F_\mathbb{A})~|~\alpha_v=I_2~\textrm{for every}~v\in\mathbf{h}\right\},\\
G_\mathbf{h}&=&\left\{
(\alpha_v)_v\in\mathrm{GL}_2(F_\mathbb{A})~|~\alpha_v=I_2~\textrm{for any}~v\in\mathbf{a}\right\}
\quad(\textrm{so},~G_\mathbb{A}=G_\mathbf{a}G_\mathbf{h}),\\
G_{\mathbb{A}+}&=&\left\{\alpha\in G_\mathbb{A}~|~\det(\alpha_v)>0~\textrm{for each}~v\in\mathbf{a}\right\},\\
G_{\mathbf{a}+}&=&\{\alpha\in G_\mathbf{a}~|~\det(\alpha_v)>0~\textrm{for every}~v\in\mathbf{a}\},\\
D&=&\left\{\alpha\in G~|~\det(\alpha)\in\mathbb{Q}^\times\right\},\\
D_\mathbb{A}&=&\left\{\alpha\in G_\mathbb{A}~|~\det(\alpha)\in\mathbb{Q}_\mathbb{A}^\times\right\},\\
D_\mathbf{a}&=&D_\mathbb{A}\cap G_\mathbf{a},\\
D_\mathbf{h}&=&D_\mathbb{A}\cap G_\mathbf{h}.
\end{eqnarray*}
Let $S_\mathbf{h}$ be a compact subgroup of $G_\mathbf{h}$ containing
an open subgroup of $D_\mathbf{h}$, and set
\begin{equation*}
S=S_\mathbf{h}G_{\mathbf{a}+}
\end{equation*}
which is a subgroup of $G_{\mathbb{A}+}$.
Put
\begin{equation*}
\Gamma_S=G\cap S
\end{equation*}
which is an arithmetic subgroup of $G$ ($\hookrightarrow\mathrm{GL}_2(\mathbb{R})^g$), and let
$k_S$ be the subfield of the maximal abelian extension $\mathbb{Q}_\mathrm{ab}$ of $\mathbb{Q}$
corresponding to $\mathbb{Q}_\mathbb{A}^\times\cap F^\times\det(S)$.
\par
It is well known that the Satake compactification
of the Hilbert modular variety $\Gamma_S\backslash\mathbb{H}^g$
is isomorphic to a normal projective variety $V_S^*$. Moreover,
$\Gamma_S\backslash\mathbb{H}^g$ is mapped onto a Zariski open subset $V_S$ of
$V_S^*$ (\cite{B-B}). Let $\varphi_S$ be the
$\Gamma_S$-invariant map $\mathbb{H}^g\rightarrow V_S$ that gives this isomorphism.
We then call $(V_S,\,\varphi_S)$ a \textit{model} of $\Gamma_S\backslash\mathbb{H}^g$.
\par
Here, we shall take a model $(V_S,\,\varphi_S)$ of $\Gamma_S\backslash\mathbb{H}^g$
as a canonical one given in \cite[Theorem 26.3]{Shimura}, and
so $V_S$ is rational over $k_S$.
We further denote by
\begin{equation*}
\mathcal{K}_S=\{g\circ\varphi_S~|~
\textrm{$g$ are $k_S$-rational functions on $V_S$ in the sense of algebraic geometry}\}.
\end{equation*}
For details of this \textit{canonical model} $(V_S,\,\varphi_S)$,
one can refer to \cite{Shimura70}.
\par
We call a point $\mathbf{w}=(w_1,\,w_2,\,\ldots,\,w_g)$ of $\mathbb{H}^g$ a
\textit{CM-point} (in Hilbert modular case) associated with the CM-type $(K,\,\{\varphi_i\}_{i=1}^g)$ if
\begin{equation*}
w_1\in K\quad\textrm{and}\quad
w_i=\varphi_i(w_1)\quad(i=1,\,2,\,\ldots,\,g).
\end{equation*}
See \cite[$\S$24.10 and Proposition 24.13]{Shimura}.

\begin{lemma}\label{ww'}
Let $\mathbf{w}=(w_1,\,w_2,\,\ldots,\,w_g)$ and $\mathbf{w}'=(w_1',\,w_2',\,\ldots,\,w_g')$ be
CM-points on $\mathbb{H}^g$. If
\begin{equation}\label{ffassumption}
f(\mathbf{w})=
f(\mathbf{w}')\quad\textrm{for all}~f\in\mathcal{K}_S~\textrm{that are
finite at}~\mathbf{w},
\end{equation}
then
$\mathbf{w}=\gamma\mathbf{w}'$ for some $\gamma\in\Gamma_S$.
\end{lemma}
\begin{proof}
Suppose that our canonical model $(V_S,\,\varphi_S)$ of
$\Gamma_S\backslash\mathbb{H}^g$ satisfies
$V_S\subseteq\mathbb{P}^m(\mathbb{C})$ and $\varphi_S=[\phi_0:\phi_1:\cdots:\phi_m]$.
Since $\varphi_S(\mathbf{w})\in\mathbb{P}^m(\mathbb{C})$, we have
\begin{equation*}
\left(
\frac{\phi_0}{\phi_i}(\mathbf{w}),\,\ldots,\,
\frac{\phi_{i-1}}{\phi_i}(\mathbf{w}),\,
\frac{\phi_{i+1}}{\phi_i}(\mathbf{w}),\,
\ldots,\,
\frac{\phi_m}{\phi_i}(\mathbf{w})
\right)\in\mathbb{A}^m(\mathbb{C})\quad\textrm{for some}~0\leq i\leq m.
\end{equation*}
Furthermore, since
\begin{equation*}
\frac{\phi_j}{\phi_i}\in\mathcal{K}_S\quad
\textrm{for all}~j=0,\,\ldots,\,i-1,\,i+1,\,\ldots,\,m,
\end{equation*}
we get by the assumption (\ref{ffassumption}) that
\begin{equation*}
\left(
\frac{\phi_0}{\phi_i}(\mathbf{w}),\,\ldots,\,
\frac{\phi_{i-1}}{\phi_i}(\mathbf{w}),\,
\frac{\phi_{i+1}}{\phi_i}(\mathbf{w}),\,
\ldots,\,
\frac{\phi_m}{\phi_i}(\mathbf{w})
\right)
=\left(
\frac{\phi_0}{\phi_i}(\mathbf{w}'),\,\ldots,\,
\frac{\phi_{i-1}}{\phi_i}(\mathbf{w}'),\,
\frac{\phi_{i+1}}{\phi_i}(\mathbf{w}'),\,
\ldots,\,
\frac{\phi_m}{\phi_i}(\mathbf{w}')
\right).
\end{equation*}
This implies that $\mathbf{w}$ and $\mathbf{w}'$ represent the same point on
$\Gamma_S\backslash\mathbb{H}^g$, and hence
$\mathbf{w}=\gamma\mathbf{w}'$ for some $\gamma\in\Gamma_S$.
\end{proof}

Let $\mathcal{K}$ be the field of all Hilbert modular functions for
some arithmetic subgroups of $G$ with Fourier coefficients in $\mathbb{Q}_\mathrm{ab}$
(\cite[$\S$25.4--25.5]{Shimura}). Furthermore, let
\begin{equation*}
\mathcal{G}=D_\mathbb{A}GG_{\mathbf{a}+}\quad\textrm{and}\quad
\mathcal{G}_+=\mathcal{G}\cap G_{\mathbb{A}+}
\end{equation*}
which are subgroups of $G_\mathbb{A}$ since
$D_\mathbb{A}$ and $G_{\mathbf{a}+}$ are normal in $G_\mathbb{A}$
(\cite[$\S$II.8]{Shimura00}).

\begin{proposition}\label{tau}
There exists a homomorphism
\begin{equation*}
\tau_\mathcal{K}~:~\mathcal{G}_+\rightarrow\mathrm{Aut}(\mathcal{K})
\end{equation*}
satisfying
\begin{enumerate}
\item[\textup{(i)}]$f^{\tau_\mathcal{K}(\alpha)}=f\circ\alpha$ for every
$f\in\mathcal{K}$ and $\alpha\in G_+$,
\item[\textup{(ii)}] $\{\gamma\in\mathcal{G}_+~|~
\tau_\mathcal{K}(\gamma)=\mathrm{id}~\textrm{on}~\mathcal{K}_S\}=
F^\times(S\cap\mathcal{G})$.
\end{enumerate}
\end{proposition}
\begin{proof}
See \cite[Theorem 26.8]{Shimura}.
\end{proof}

By $\mathcal{A}_0(\Gamma_S,\,k_S)$ we mean the
field of Hilbert modular functions for $\Gamma_S$ with Fourier coefficients in $k_S$, and let
\begin{equation*}
\Delta_S=\left\{
\begin{bmatrix}1&0\\0&t\end{bmatrix}~|~
t\in\hspace{-0.2cm}\prod_{p\,:\,\mathrm{primes}}\hspace{-0.2cm}\mathbb{Z}_p^\times~\textrm{such that}~
[t,\,\mathbb{Q}]=\mathrm{id}~\mathrm{on}~k_S
\right\},
\end{equation*}
where $[\,\cdot\,,\,\mathbb{Q}]:\mathbb{Q}_\mathbb{A}^\times\rightarrow
\mathrm{Gal}(\mathbb{Q}_\mathrm{ab}/\mathbb{Q})$ is the Artin map for $\mathbb{Q}$.

\begin{lemma}\label{AK}
If $\Delta_S\subseteq S$, then $\mathcal{A}_0(\Gamma_S,\,k_S)=\mathcal{K}_S$.
\end{lemma}
\begin{proof}
See \cite[Theorem 26.4]{Shimura}.
\end{proof}

Now, we set
\begin{equation*}
\Delta=\left\{\begin{bmatrix}1&0\\0&\xi\end{bmatrix}~|~\xi\in\mathcal{O}_F^\times\hspace{-0.2cm}\prod_{p\,:\,\mathrm{primes}}
\hspace{-0.3cm}\mathbb{Z}_p^\times\right\}
\end{equation*}
and
\begin{equation}\label{Sh}
S_\mathbf{h}=\Delta\left\{\beta=(\beta_v)_v\in D_\mathbf{h}~\Bigg|~
\begin{array}{l}
\textrm{for every}~v\in\mathbf{h},~\beta_v\in\mathrm{GL}_2(\mathcal{O}_{F,\,v})~
\textrm{and}~\\
\beta_v\equiv\begin{bmatrix}\mu&\mathrm{*}\\0&\mathrm{*}\end{bmatrix}
\Mod{M_2(N\mathcal{O}_{F,\,v})}~
\textrm{for some}~\mu\in\mathcal{O}_F^\times
\end{array}
\right\}.
\end{equation}
Here, a unit $\mu$ in $\mathcal{O}_F$ is given for an element $\beta$ of $D_\mathbf{h}$ so that it does not vary
according to each component $\beta_v$.

\begin{proposition}\label{GkA}
If we let $S=S_\mathbf{h}G_{\mathbf{a}+}$ with $S_\mathbf{h}$ stated in \textup{(\ref{Sh})}, then we have
\begin{enumerate}
\item[\textup{(i)}] $\Gamma_S=\Gamma_{F,\,1}^+(N)$ and
$k_S=\mathbb{Q}$,
\item[\textup{(ii)}] $\mathcal{A}_0(S,\,k_S)=\mathcal{K}_S$.
\end{enumerate}
\end{proposition}
\begin{proof}
\begin{enumerate}
\item[(i)] Let $\alpha\in\Gamma_S=G\cap S$,
and so $\det(\alpha)\gg0$ and
\begin{equation*}
\alpha_v=\begin{bmatrix}a_1&a_2\\a_3&a_4\end{bmatrix}=
\begin{bmatrix}1&0\\0&\xi_v\end{bmatrix}\beta_v\quad(v\in\mathbf{h})
\end{equation*}
for some $\displaystyle\xi\in\mathcal{O}_F^\times\prod_p\mathbb{Z}_p^\times$ and $\beta\in D_\mathbf{h}$ such that
\begin{equation}\label{bb}
\beta_v\in\mathrm{GL}_2(\mathcal{O}_{F,\,v})\quad\textrm{and}\quad
\beta_v\equiv\begin{bmatrix}\mu&\mathrm{*}\\
0&\mathrm{*}\end{bmatrix}\Mod{M_2(N\mathcal{O}_{F,\,v})}~
\textrm{for some}~\mu\in\mathcal{O}_F^\times.
\end{equation}
Since
\begin{equation*}
\alpha_v=\begin{bmatrix}1&0\\0&\xi_v\end{bmatrix}
\begin{bmatrix}a_1&a_2\\a_3/\xi_v&a_4/\xi_v\end{bmatrix}
\end{equation*}
and $v\in\mathbf{h}$ is arbitrary,
we get by (\ref{bb}) that
\begin{equation}\label{a1234}
a_1,\,a_2,\,a_3,\,a_4\in\mathcal{O}_F\quad\textrm{and}\quad
a_1\equiv\mu,\,a_3\equiv0\Mod{N\mathcal{O}_F}.
\end{equation}
Applying the same argument to $\alpha^{-1}\in\Gamma_S$,
we derive that
every entry of $\alpha^{-1}$ belongs to $\mathcal{O}_F$, and hence
$\det(\alpha)\in\mathcal{O}_F^\times$. Since
\begin{equation*}
\det(\alpha)\equiv a_1a_4-a_2a_3\equiv a_1a_4\equiv\mu a_4\Mod{N\mathcal{O}_F}
\end{equation*}
by (\ref{a1234}), we obtain
\begin{equation*}
a_4\equiv\frac{\det(\alpha)}{\mu}\Mod{N\mathcal{O}_F}\quad
\textrm{with}~\frac{\det(\alpha)}{\mu}\in\mathcal{O}_F^\times,
\end{equation*}
which holds $\Gamma_S\subseteq\Gamma_{F,\,1}^+(N)$.
\par
Now, let $\gamma=\begin{bmatrix}c_1&c_2\\c_3&c_4\end{bmatrix}\in\Gamma_{F,\,1}^+(N)$, and so
$\det(\gamma)\gg0$ and
\begin{equation}\label{g0z}
\gamma\equiv\begin{bmatrix}\mathrm{*}&\mathrm{*}\\
0&\zeta\end{bmatrix}\Mod{M_2(N\mathcal{O}_F)}~\textrm{for some}~\zeta\in\mathcal{O}_F^\times.
\end{equation}
Letting $\xi=\det(\gamma)$ ($\in\mathcal{O}_F^\times$), we attain
\begin{equation*}
\gamma=\begin{bmatrix}1&0\\0&\xi\end{bmatrix}
\delta\quad
\textrm{with}~\delta=\begin{bmatrix}c_1&c_2\\c_3/\xi&c_4/\xi\end{bmatrix}\in\mathrm{SL}_2(\mathcal{O}_F).
\end{equation*}
It then follows from (\ref{g0z}) that
\begin{equation*}
\beta_v\equiv\begin{bmatrix}
\xi/\zeta&\mathrm{*}\\0&\zeta/\xi\end{bmatrix}\Mod{M_2(N\mathcal{O}_{F,\,v})}\quad(v\in\mathbf{h}),
\end{equation*}
which shows that $\gamma$ belongs to $\Gamma_S$.
Therefore we achieve
$\Gamma_{F,\,1}^+(N)\subseteq\Gamma_S$, and hence
$\Gamma_S=\Gamma_{F,\,1}^+(N)$.
Moreover, since
$\mathbb{Q}_\mathbb{A}^\times\cap F^\times\det(S)=\mathbb{Q}_\mathbb{A}^\times$, we conclude
$k_S=\mathbb{Q}$.
\item[(ii)] Since $\Delta_S\subseteq S$, we get $\mathcal{A}_0(\Gamma_S,\,k_S)=\mathcal{K}_S$ by Lemma \ref{AK}.
\end{enumerate}
\end{proof}

\section {Shimura's reciprocity law}

In this section, we shall briefly introduce Shimura's reciprocity law applied to
singular values of Hilbert modular functions, from which we will be able to analyze the behavior of the invariant $f(C)$ defined in $\S$\ref{defininginvariant} under
a certain Galois group.
\par
Take any finite Galois extension $L$ of $\mathbb{Q}$ containing $K$. Let
\begin{eqnarray*}
T&=&\{\sigma\in\mathrm{Gal}(L/\mathbb{Q})~|~\sigma|_K=\varphi_i~
\textrm{for some}~i=1,\,2,\,\ldots,\,g\},\\
T^*&=&\{\sigma^{-1}~|~\sigma\in T\},\\
H^*&=&\{\gamma\in\mathrm{Gal}(L/\mathbb{Q})~|~\gamma T^*=T^*\}.
\end{eqnarray*}
Let $K^*$ be the subfield of $L$ corresponding to $H^*$ with
$\{\psi_j\}_{j=1}^n$ the set of all embeddings of $K^*$ into $\mathbb{C}$ obtained from the elements of $T^*$.
Then, it is well known that $(K^*,\,\{\psi_j\}_{j=1}^n)$ is a primitive CM-type
and
\begin{equation*}
K^*=\mathbb{Q}\left(\sum_{i=1}^g\varphi_i(a)~|~a\in K\right).
\end{equation*}
This CM-type is independent of the choice of $L$, and
called the \textit{reflex} of $(K,\,\{\varphi_i\}_{i=1}^g)$
(\cite[Proposition 28 in Chapter II]{Shimura}).
We consider a group homomorphism
\begin{equation*}
\begin{array}{ccccc}
g&:&(K^*)^\times&\rightarrow&K^\times\\
&&d&\mapsto&\displaystyle\prod_{j=1}^n\psi_j(d)
\end{array}
\end{equation*}
which can be continuously extended to the homomorphism
\begin{equation*}
g~:~(K^*)_\mathbb{A}^\times~\rightarrow~K_\mathbb{A}^\times
\end{equation*}
of idele groups. Furthermore, we have a homomorphism
\begin{equation*}
\begin{array}{ccccl}
\mathfrak{g}&:&I_{K^*}(\mathcal{O}_{K^*})&\rightarrow&I_K(\mathcal{O}_K)\\
&&\mathfrak{a}&\mapsto&~~\mathfrak{g}(\mathfrak{a})\quad\textrm{satisfying}\quad
\displaystyle
\mathfrak{g}(\mathfrak{a})\mathcal{O}_L=\prod_{j=1}^n\psi_j(\mathfrak{a})\mathcal{O}_L
\end{array}
\end{equation*}
(\cite[Proposition 29 in Chapter II]{Shimura}). Here we observe that
\begin{equation}\label{gdgd}
\mathfrak{g}(d\mathcal{O}_{K^*})=g(d)\mathcal{O}_K
\quad\textrm{for all}~d\in(K^*)^\times.
\end{equation}

\begin{lemma}\label{gd1}
If $d$ is an element of $(K^*)^\times$ such that
$d\equiv^*1\Mod{N\mathcal{O}_{K^*}}$, then
\begin{equation*}
g(d)\equiv^*1\Mod{N\mathcal{O}_K}.
\end{equation*}
\end{lemma}
\begin{proof}
Since $d\equiv^*1\Mod{N\mathcal{O}_{K^*}}$,
there exist $M\in\mathbb{Z}_{>0}$ and $\nu\in\mathcal{O}_K$ such that
\begin{equation*}
\gcd(M,\,N)=1\quad\textrm{and}\quad
M(d-1)=N\nu.
\end{equation*}
We then achieve that
\begin{equation*}
M^ng(d)=M^ng(1+M^{-1}N\nu)=M^n\prod_{j=1}^n\psi_j(1+M^{-1}N\nu)=
\prod_{j=1}^n\left\{
M+N\psi_j(\nu)\right\}\equiv M^n\Mod{N\mathcal{O}_K},
\end{equation*}
and hence
\begin{equation*}
g(d)\equiv^*1\Mod{N\mathcal{O}_K}.
\end{equation*}

\end{proof}

By Lemma \ref{gd1} and (\ref{gdgd}), we obtain a homomorphism
\begin{equation*}
\begin{array}{ccccc}
\mathfrak{g}_N&:&\mathcal{C}(N\mathcal{O}_{K^*})&\rightarrow&\mathcal{C}(N\mathcal{O}_K)\\
&&\mathrm{[}\mathfrak{a}\mathrm{]}&\mapsto&
\mathrm{[}\mathfrak{g}(\mathfrak{a})\mathrm{]}
\end{array}
\end{equation*}
of ray class groups.
For the element $\xi$ of $K$ described in $\S$\ref{defininginvariant}, let
\begin{equation*}
h_\xi~:~K~\rightarrow~M_2(F)
\end{equation*}
be the regular representation with respect to the ordered basis $\{\xi,\,1\}$ of $K$ over $F$, namely,
$h_\xi$ is defined by the relation
\begin{equation*}
\nu\begin{bmatrix}\xi\\1\end{bmatrix}=h_\xi(\nu)\begin{bmatrix}\xi\\1\end{bmatrix}\quad(\nu\in K).
\end{equation*}

\begin{proposition}[Shimura's reciprocity law]\label{reciprocity}
Let $f=f(z_1,\,z_2,\,\ldots,\,z_g)$ be a
Hilbert modular function for an arithmetic subgroup of $\mathrm{GL}_2(F)$ with
Fourier coefficients in $\mathbb{Q}_\mathrm{ab}$. Suppose that $f$ is finite at $(\varphi_1(\xi),\,\varphi_2(\xi),\,\ldots,\,\varphi_g(\xi))$. Then,
\begin{enumerate}
\item[\textup{(i)}] $f(\varphi_1(\xi),\,\varphi_2(\xi),\,\ldots,\,
\varphi_g(\xi))$ belongs to the maximal abelian extension $K^*_\mathrm{ab}$ of $K^*$.
\item[\textup{(ii)}] If $s\in(K^*)_\mathbb{A}^\times$, then
\begin{equation*}
f(\varphi_1(\xi),\,\varphi_2(\xi),\,\ldots,\,
\varphi_g(\xi))^{[s,\,K^*]}=
f^{\tau_\mathcal{K}(h_\xi(g(s^{-1})))}(\varphi_1(\xi),\,\varphi_2(\xi),\,\ldots,\,
\varphi_g(\xi)),
\end{equation*}
where $[\,\cdot\,,\,K^*]:(K^*)_\mathbb{A}^\times\rightarrow\mathrm{Gal}(K^*_\mathrm{ab}/K^*)$ is the Artin map for $K^*$.
\end{enumerate}
\end{proposition}
\begin{proof}
See \cite[Proposition 24.13 and Theorem 26.8 (4)]{Shimura}.
\end{proof}

\section {Generation of class fields over the reflex field}

Finally, we shall prove our second main theorem
which shows that
the singular values of Hilbert modular functions in $\mathcal{A}_0(\Gamma_{F,\,1}^+(N),\,\mathbb{Q})$
generate the subfield of $K^*_{(N)}$ corresponding to
$\mathrm{Ker}(\mathfrak{g}_N)$ over $K^*$.

\begin{lemma}\label{DA+}
If $s\in(K^*)_\mathbb{A}^\times$, then $h_\xi(g(s^{-1}))\in D_{\mathbb{A}+}$.
\end{lemma}
\begin{proof}
See \cite[p. 172]{Shimura}.
\end{proof}

\begin{proposition}\label{transform}
Let $f\in\mathcal{A}_0(\Gamma_{F,\,1}^+(N),\,\mathbb{Q})$ and $C\in\mathcal{C}(N\mathcal{O}_K)$.
Furthermore, let $\mathfrak{a}$ be a nontrivial ideal of $\mathcal{O}_{K^*}$ relatively prime to $N\mathcal{O}_{K^*}$,
and let $C'=\mathfrak{g}_N([\mathfrak{a}])$ in $\mathcal{C}(N\mathcal{O}_K)$.
If $f(C)$ is finite, then
it lies in $K^*_{(N)}$ and satisfies
\begin{equation*}
f(C)^{\sigma_N([\mathfrak{a}])}=f(CC'),
\end{equation*}
where $\sigma_N:\mathcal{C}(N\mathcal{O}_{K^*})\rightarrow
\mathrm{Gal}(K^*_{(N)}/K^*)$ is the Artin map.
\end{proposition}
\begin{proof}
Note first that $f(C)$ belongs to $K^*_\mathrm{ab}$ by Definition \ref{rayclassinvariant}, Propositions \ref{tau} (i) and  \ref{reciprocity} (i).
Let $\mathfrak{c}$ be an integral ideal in the ray class $C$,
and let $\xi_1$ and $\xi_2$ be elements of $K$ such that
$\mathfrak{c}^{-1}=[\xi_1,\,\xi_2]_F$
satisfying (\ref{rhx}) in Lemma \ref{x1x2}.
Also by Lemma \ref{x1x2} one can take elements $\xi_1''$ and $\xi_2''$ of $K$ satisfying
\begin{equation}\label{cgax''}
(\mathfrak{c}\mathfrak{g}(\mathfrak{a}))^{-1}=[\xi_1'',\,\xi_2'']_F
\end{equation}
and
\begin{equation*}
\varphi_i(\xi'')\in\mathbb{H}\quad\textrm{with}~\xi''=\frac{\xi_1''}{\xi_2''}\quad(i=1,\,2,\,\ldots,\,g).
\end{equation*}
Since $\mathfrak{c}^{-1}\subseteq(\mathfrak{c}\mathfrak{g}(\mathfrak{a}))^{-1}$, we have
\begin{equation}\label{xBx''}
\begin{bmatrix}\xi_1\\\xi_2\end{bmatrix}=B\begin{bmatrix}\xi_1''\\\xi_2''\end{bmatrix}\quad
\textrm{for some}~B\in M_2(\mathcal{O}_F).
\end{equation}
Here, one can show by (\ref{rhx}) and
(\ref{xBx''})
that $\det(B)\gg0$ in a similar way to Lemma \ref{det>>0} (i).
It then follows from (\ref{wAx}) that
\begin{equation*}
\begin{bmatrix}
\omega_K\\1
\end{bmatrix}=AB\begin{bmatrix}
\xi_1''\\\xi_2''
\end{bmatrix}\quad\textrm{with}~A\in M_2(\mathcal{O}_F).
\end{equation*}
Thus, if we let $(AB)_1$ be a matrix in $\mathrm{SL}_2(\mathcal{O}_F)$ such that
\begin{equation*}
AB\equiv\begin{bmatrix}1&0\\0&\det(AB)\end{bmatrix}
(AB)_1\Mod{M_2(N\mathcal{O}_F)},
\end{equation*}
then we get
\begin{equation}\label{fCC'}
f(CC')=f^{\tau_\mathcal{K}((AB)_1)}(\varphi_1(\xi''),\,\varphi_2(\xi''),\,
\ldots,\,\varphi_g(\xi''))
\end{equation}
by Definition \ref{rayclassinvariant} and Proposition \ref{tau} (i).
Moreover, if we let $(A_1B)_1$ be a matrix in $\mathrm{SL}_2(\mathcal{O}_F)$
which satisfies
\begin{equation*}
A_1B\equiv\begin{bmatrix}1&0\\
0&\det(A_1B)\end{bmatrix}(A_1B)_1
\equiv\begin{bmatrix}1&0\\
0&\det(B)\end{bmatrix}(A_1B)_1
\Mod{M_2(N\mathcal{O}_F)},
\end{equation*}
then we derive that
\begin{equation*}
\begin{bmatrix}1&0\\0&\det(AB)\end{bmatrix}(AB)_1\equiv AB\equiv
\begin{bmatrix}1&0\\0&\det(A)\end{bmatrix}A_1B\equiv
\begin{bmatrix}1&0\\0&\det(A)\end{bmatrix}\begin{bmatrix}1&0\\0&\det(B)\end{bmatrix}(A_1B)_1
\Mod{M_2(N\mathcal{O}_F)},
\end{equation*}
from which we conclude
\begin{equation}\label{1relation}
(AB)_1\equiv(A_1B)_1\Mod{M_2(N\mathcal{O}_F)}.
\end{equation}
\par
For $\displaystyle\widehat{\mathbb{Z}}=\hspace{-0.2cm}\prod_{p\,:\,\textrm{primes}}\hspace{-0.2cm}\mathbb{Z}_p$, denote $\widehat{K^*}^\times=(K^*\otimes\widehat{\mathbb{Z}})^\times$
which is isomorphic to the finite part of $(K^*)_\mathbb{A}^\times$.
Let $s=(s_p)_p$ be an idele in $\widehat{K^*}^\times$ such that
\begin{equation}\label{sp}
\left\{\begin{array}{ll}
s_p=1&\textrm{if}~p\,|\,N,\\
s_p\mathcal{O}_{K^*,\,p}=\mathfrak{a}_p&\textrm{if}~p\nmid N.
\end{array}\right.
\end{equation}
Since $\mathfrak{a}$ is relatively prime to $N\mathcal{O}_{K^*}$, we attain by (\ref{sp}) that
\begin{equation}\label{sOa}
s_p^{-1}\mathcal{O}_{K^*,\,p}=\mathfrak{a}_p^{-1}\quad\textrm{for every prime}~p.
\end{equation}
We then see by the definition of the map $h_{\xi,\,p}$ that
\begin{equation*}
h_{\xi,\,p}(g(s_p^{-1}))\begin{bmatrix}\xi_1\\\xi_2\end{bmatrix}=
\xi_2h_{\xi,\,p}(g(s_p^{-1}))\begin{bmatrix}\xi\\1\end{bmatrix}=
\xi_2g(s_p^{-1})\begin{bmatrix}\xi\\1\end{bmatrix}=
g(s_p^{-1})\begin{bmatrix}\xi_1\\\xi_2\end{bmatrix},
\end{equation*}
which implies by the fact $\mathfrak{c}^{-1}=[\xi_1,\,\xi_2]_F$ and (\ref{sOa}) that $h_{\xi,\,p}(g(s_p^{-1}))\begin{bmatrix}\xi_1\\\xi_2\end{bmatrix}$
is a $\mathcal{O}_{F,\,p}$-basis for $(\mathfrak{c}\mathfrak{g}(\mathfrak{a}))_p^{-1}$ for every prime $p$.
On the other hand, we observe by (\ref{cgax''}) and (\ref{xBx''}) that
$B^{-1}\begin{bmatrix}
\xi_1\\\xi_2
\end{bmatrix}$ is also a $\mathcal{O}_{F,\,p}$-basis for $(\mathfrak{c}\mathfrak{g}(\mathfrak{a}))_p^{-1}$. Therefore we obtain
\begin{equation}\label{hgsgB}
h_{\xi,\,p}(g(s_p^{-1}))=\gamma_p B^{-1}
\quad\textrm{for some}~\gamma_p\in\mathrm{GL}_2(\mathcal{O}_{F,\,p}),
\end{equation}
and so
\begin{equation}\label{hgsgBg}
h_\xi(g(s^{-1}))=\gamma B^{-1}\quad\textrm{where}~
\gamma=(\gamma_p)_p\in\prod_p\mathrm{GL}_2(\mathcal{O}_{F,\,p}).
\end{equation}
Here, $\gamma B^{-1}$ belongs to $D_{\mathbb{A}+}$ by Lemma \ref{DA+}, and hence
\begin{equation}\label{inG}
\gamma\in\mathcal{G}.
\end{equation}
Note that for every prime $p$ dividing $N$, we get by the condition $s_p=1$ in (\ref{sp}) and (\ref{hgsgB}) that
$\gamma_p B^{-1}=I_2$.
We then deduce that for every prime $p$
\begin{equation*}
A_1\gamma_p\equiv A_1B\Mod{M_2(N\mathcal{O}_{F,\,p})}
\equiv\begin{bmatrix}1&0\\0&\det(B)\end{bmatrix}(A_1B)_1\Mod{M_2(N\mathcal{O}_{F,\,p})},
\end{equation*}
which shows by (\ref{Sh}) and (\ref{inG})
\begin{equation}\label{inS}
(A_1\gamma)(A_1B)_1^{-1}\in S\cap\mathcal{G}.
\end{equation}
Now, we derive that
\begin{eqnarray*}
f(C)^{[s,\,K^*]}&=&f^{\tau_\mathcal{K}(A_1)}(\varphi_1(\xi),\,\varphi_2(\xi),\,\ldots,\,\varphi_g(\xi))^{[s,\,K^*]}\quad\textrm{by
Definition \ref{rayclassinvariant} and Proposition \ref{tau} (i)}\\
&=&f^{\tau_\mathcal{K}(A_1)\tau_\mathcal{K}(h_\xi(g(s^{-1})))}(\varphi_1(\xi),\,\varphi_2(\xi),\,\ldots,\,\varphi_g(\xi))\quad
\textrm{by Proposition \ref{reciprocity}}\\
&=&f^{\tau_\mathcal{K}(A_1)\tau_\mathcal{K}(\gamma B^{-1})}(\varphi_1(\xi),\,\varphi_2(\xi),\,\ldots,\,\varphi_g(\xi))
\quad\textrm{by (\ref{hgsgBg})}\\
&=&(f^{\tau_\mathcal{K}(A_1)\tau_\mathcal{K}(\gamma)}\circ B^{-1})(\varphi_1(\xi),\,\varphi_2(\xi),\,\ldots,\,\varphi_g(\xi))\quad
\textrm{by Proposition \ref{tau} (i)}\\
&=&f^{\tau_\mathcal{K}(A_1\gamma)}(\varphi_1(B^{-1}\xi),\,\varphi_2(B^{-1}\xi),\,\ldots,\,\varphi_g(B^{-1}\xi))\\
&=&f^{\tau_\mathcal{K}(A_1\gamma)}(\varphi_1(\xi''),\,\varphi_2(\xi''),\,\ldots,\,\varphi_g(\xi''))\quad
\textrm{by (\ref{xBx''})}\\
&=&f^{\tau_\mathcal{K}((A_1B)_1)}(\varphi_1(\xi''),\,\varphi_2(\xi''),\,\ldots,\,\varphi_g(\xi''))
\quad\textrm{by (\ref{inS}), Propositions \ref{tau} (ii) and \ref{GkA}}\\
&=&f^{\tau_\mathcal{K}((AB)_1)}(\varphi_1(\xi''),\,\varphi_2(\xi''),\,\ldots,\,\varphi_g(\xi''))\\
&&\hspace{4.5cm}\textrm{by (\ref{1relation})
and the fact that $f$ is modular for $\Gamma_{F,\,1}^+(N)$}\\
&=&f(CC')\quad\textrm{by (\ref{fCC'})}.
\end{eqnarray*}
\par
By Proposition \ref{reciprocity} (i) we may take a sufficiently large positive integer $M$ so that
$f(C)\in K^*_{(M)}$. Suppose that $\sigma_{M}([\mathfrak{a}])$ leaves
$K^*_{(N)}$ fixed elementwise. Since
\begin{equation*}
\mathrm{Gal}(K^*_{(M)}/K^*_{(N)})\simeq P_{K^*,\,1}(N\mathcal{O}_{K^*})
/P_{K^*,\,1}(M\mathcal{O}_{K^*}),
\end{equation*}
we have
\begin{equation*}
\mathfrak{a}=d\mathcal{O}_{K^*}\quad\textrm{for some}~d\in\mathcal{O}_{K^*}~
\textrm{such that}~d\equiv1\Mod{N\mathcal{O}_{K^*}}.
\end{equation*}
Thus we achieve by (\ref{gdgd}) and Lemma \ref{gd1} that
\begin{equation*}
C'=\mathfrak{g}_N([\mathfrak{a}])=[\mathfrak{g}(d\mathcal{O}_{K^*})]
=[g(d)\mathcal{O}_K]=[\mathcal{O}_K]\quad\textrm{in}~\mathcal{C}(N\mathcal{O}_K),
\end{equation*}
and so
\begin{equation*}
f(C)^{\sigma_{M}([\mathfrak{a}])}=f(C[\mathcal{O}_K])=f(C).
\end{equation*}
This yields by the Galois theory that $f(C)$ lies in $K^*_{(N)}$.
\end{proof}

\begin{remark}
For $\begin{bmatrix}r_1&r_2\end{bmatrix}\in M_{1,\,2}(\mathbb{Q})
\setminus M_{1,\,2}(\mathbb{Z})$, the Siegel function
$g_{\left[\begin{smallmatrix}r_1&r_2\end{smallmatrix}\right]}(\tau)$
is given by the infinite product expansion
\begin{eqnarray*}
g_{\left[\begin{smallmatrix}r_1&r_2\end{smallmatrix}\right]}(\tau)
&=&-e^{\pi\mathrm{i}r_2(r_1-1)}
q^{\frac{1}{2}(r_1^2-r_1+\frac{1}{6})}
(1-q^{r_1}e^{2\pi\mathrm{i}r_2})\\
&&\times
\prod_{n=1}^\infty(1-q^{n+r_1}e^{2\pi\mathrm{i}r_2})
(1-q^{n-r_1}e^{-2\pi\mathrm{i}r_2})\quad(q=e^{2\pi\mathrm{i}\tau},~\tau\in\mathbb{H}).
\end{eqnarray*}
Suppose that $K$ is an imaginary quadratic field.
Then $K^*$ is nothing but $K$ itself.
Let $\mathfrak{n}$ be
a proper nontrivial ideal of $\mathcal{O}_K$ in which
$N$ is the least positive integer, and let $C\in\mathcal{C}(\mathfrak{n})$. Take any
ideal $\mathfrak{c}$ of $\mathcal{O}_K$ belonging to the ray class $C$,
and take $\xi_1,\xi_2\in K^\times$ such that
\begin{equation*}
\mathfrak{n}\mathfrak{c}^{-1}=[\xi_1,\,\xi_2]\quad
\textrm{and}\quad\xi=\frac{\xi_1}{\xi_2}\in\mathbb{H}.
\end{equation*}
Since $N$ is contained in $\mathfrak{n}\mathfrak{c}^{-1}$, it can be written as
\begin{equation*}
N=a\xi_1+b\xi_2\quad\textrm{for some}~a,\,b\in\mathbb{Z}.
\end{equation*}
Then, the Siegel-Ramachandra invariant $g(C)$ is defined by
\begin{equation*}
g(C)=g_{\left[\begin{smallmatrix}\frac{a}{N}&\frac{b}{N}\end{smallmatrix}\right]}(\xi)^{12N}
\end{equation*}
which is independent of the choice of $\mathfrak{c}$, $\xi_1$ and $\xi_2$.
This invariant lies in the ray class field $K_\mathfrak{n}$ and satisfies
\begin{equation*}
g(C)^{\sigma_\mathfrak{n}(C')}=g(CC')\quad(C'\in\mathcal{C}(\mathfrak{n})),
\end{equation*}
where $\sigma_\mathfrak{n}:\mathcal{C}(\mathfrak{n})\rightarrow\mathrm{Gal}(K_\mathfrak{n}/K)$ is the Artin map.
For the detailed arithmetic and algebraic properties of Siegel functions
and Siegel-Ramachandra invariants, one can refer to \cite{K-L}
or \cite{Ramachandra}, \cite{Siegel}.
\end{remark}

\begin{lemma}\label{fCfC'}
If $C,\,C'\in\mathcal{C}(N\mathcal{O}_K)$ such that $C\neq C'$,
then there is $f\in\mathcal{A}_0(\Gamma_{F,\,1}^+(N),\,\mathbb{Q})$ such that
$f(C)$ is finite and
$f(C)\neq f(C')$.
\end{lemma}
\begin{proof}
By Theorem \ref{main1} and Lemma \ref{+representative}, we have
\begin{equation*}
C=\phi_{K,\,N}([Q])\quad\textrm{and}\quad C'=\phi_{K,\,N}([Q'])
\end{equation*}
for some $Q,\,Q'\in\mathcal{Q}_F^+(N,\,d_K)$
such that $[Q]\neq[Q']$ in $\mathcal{C}_F(N,\,d_K)$.
Suppose on the contrary that
\begin{equation*}
f(C)=f(C')\quad\textrm{for all}~f\in\mathcal{A}_0(\Gamma_{F,\,1}^+(N),\,\mathbb{Q})~
\textrm{such that $f(C)$ are finite}.
\end{equation*}
Then we get by Lemmas \ref{f(w)} and \ref{ww'} that
\begin{equation*}
(\varphi_1(-\overline{\omega_Q}),\,\varphi_2(-\overline{\omega_Q}),\,\ldots,\,
\varphi_g(-\overline{\omega_Q}))=
\gamma
(\varphi_1(-\overline{\omega_{Q'}}),\,\varphi_2(-\overline{\omega_{Q'}}),\,\ldots,\,\varphi_g(-\overline{\omega_{Q'}}))
\end{equation*}
for some $\gamma\in\Gamma_{F,\,1}^+(N)$.
Since $\varphi_1=\mathrm{id}_K$, we get
$-\overline{\omega_Q}=\gamma(-\overline{\omega_{Q'}})$, and hence $[Q]=[Q']$ in
$\mathcal{C}_F(N,\,d_K)$ by Lemma \ref{qqcwgw}. But, this contradicts $C\neq C'$.
Therefore we conclude that
there exists $f\in\mathcal{A}_0(\Gamma_{F,\,1}^+(N),\,\mathbb{Q})$ such that
$f(C)$ is finite and $f(C)\neq f(C')$.
\end{proof}

\begin{theorem}\label{main2}
The field
\begin{equation*}
L_C=K^*(f(C)~|~f\in\mathcal{A}_0(\Gamma_{F,\,1}^+(N),\,\mathbb{Q})~
\textrm{such that}~f(C)~\textrm{is finite})
\end{equation*}
is a subfield of the ray class field $K^*_{(N)}$ for which
\begin{equation*}
\mathrm{Gal}(L_C/K^*)\simeq\mathcal{C}(N\mathcal{O}_{K^*})/\mathrm{Ker}(\mathfrak{g}_N).
\end{equation*}
Moreover, it is independent of $C$.
\end{theorem}
\begin{proof}
Let $\sigma_N:\mathcal{C}(N\mathcal{O}_{K^*})\rightarrow
\mathrm{Gal}(K^*_{(N)}/K^*)$ be the Artin map.
For a nontrivial ideal $\mathfrak{a}$ of $\mathcal{O}_{K^*}$ relatively prime to $N\mathcal{O}_{K^*}$,
we deduce that
\begin{eqnarray*}
&&\textrm{$\sigma_N([\mathfrak{a}])$ leaves $L_C$ fixed elementwise}\\
&\Longleftrightarrow&
f(C)^{\sigma_N([\mathfrak{a}])}=f(C)\quad
\textrm{for every $f\in\mathcal{A}_0(\Gamma_{F,\,1}^+(N),\,\mathbb{Q})$
such that $f(C)$ is finite}\\
&\Longleftrightarrow&
f(C\mathfrak{g}_N([\mathfrak{a}]))=f(C)\quad
\textrm{for each $f\in\mathcal{A}_0(\Gamma_{F,\,1}^+(N),\,\mathbb{Q})$
such that $f(C)$ is finite}\\
&&\hspace{9.7cm}\textrm{by Proposition \ref{transform}}\\
&\Longleftrightarrow&C\mathfrak{g}_N([\mathfrak{a}])=C\quad\textrm{by Lemma \ref{fCfC'}}\\
&\Longleftrightarrow&\mathfrak{g}_N([\mathfrak{a}])=[\mathcal{O}_K]\\
&\Longleftrightarrow&[\mathfrak{a}]\in\mathrm{Ker}(\mathfrak{g}_N).
\end{eqnarray*}
Therefore we achieve by the Galois theory that $L_C$ is independent of $C$ and
\begin{equation*}
\mathrm{Gal}(K^*_{(N)}/L_C)\simeq\mathrm{Ker}(\mathfrak{g}_N).
\end{equation*}
Moreover, since $K^*_{(N)}/K^*$ is an abelian extension, we obtain
\begin{equation*}
\mathrm{Gal}(L_C/K^*)\simeq\mathcal{C}(N\mathcal{O}_{K^*})/\mathrm{Ker}(\mathfrak{g}_N).
\end{equation*}
\end{proof}

\begin{remark}
In particular, if $K$ is an imaginary quadratic field, then $K^*=K$ and
$\mathfrak{g}_N$ is the identity map on $\mathcal{C}(N\mathcal{O}_K)$.
And, Assumption \ref{narrowclassnumberone} is automatically satisfied
because $F=\mathbb{Q}$. Let
$d_K$ ($<0$) be the discriminant of $K$, and set
\begin{equation*}
\omega_K=\left\{
\begin{array}{cl}
\displaystyle\frac{-1+\sqrt{d_K}}{2} & \textrm{if}~d_K\equiv1\Mod{4},\vspace{0.1cm}\\
\displaystyle\frac{\sqrt{d_K}}{2} & \textrm{if}~d_K\equiv0\Mod{4}
\end{array}
\right.
\end{equation*}
which satisfies $\mathcal{O}_K=\mathbb{Z}\omega_K+\mathbb{Z}$
(\cite[Theorem 9.2 in Chapter I]{Janusz}).
Then, Theorem \ref{main2} implies that
\begin{equation*}
K_{(N)}=K(f(-\overline{\omega_K})~|~f\in\mathbb{Q}(X_1(N))~\textrm{is finite
at $-\overline{\omega_K}$}),
\end{equation*}
where $\mathbb{Q}(X_1(N))$ is the field of
meromorphic functions
with rational Fourier coefficients
on the modular curve $X_1(N)$ for the congruence subgroup
\begin{equation*}
\Gamma_1(N)=\left\{\gamma\in\mathrm{SL}_2(\mathbb{Z})~|~
\gamma\equiv\begin{bmatrix}1&\mathrm{*}\\
0&1\end{bmatrix}\Mod{M_2(N\mathbb{Z})}\right\}.
\end{equation*}
Furthermore, Theorem \ref{main1} and Proposition \ref{transform} give rise to an isomorphism
\begin{equation}\label{revisit}
\begin{array}{ccl}
\mathcal{C}_\mathbb{Q}(N,\,d_K)&\stackrel{\sim}{\rightarrow}&\mathrm{Gal}(K_{(N)}/K)\\
\mathrm{[}Q\mathrm{]}&\mapsto&
\left(f(-\overline{\omega_K})\mapsto
f(-\overline{\omega_Q})~|~
f\in\mathbb{Q}(X_1(N))~\textrm{is finite at $-\overline{\omega_K}$}\right).
\end{array}
\end{equation}
Note that the case $N=1$ in (\ref{revisit}) is a revisit of the classical result for $\mathcal{O}=\mathcal{O}_K$ in (\ref{CDGHK}).
The case $N\geq2$ in (\ref{revisit}) was recently explained by Jung, Koo and Shin (\cite{J-K-S}).
\end{remark}

\bibliographystyle{amsplain}

\begin{thebibliography}{99}

\bibitem {Aluffi} P. Aluffi, \textit{Algebra: chapter 0},
Grad. Studies in Math. 104,  Amer. Math. Soc., Providence,
R. I., 2009.

\bibitem {B-B} W. L. Baily, Jr. and A. Borel,
\textit{Compactification of arithmetic quotients
of bounded symmetric domains}, Ann. of Math. (2) 84 (1966), 442--528.

\bibitem {BhargavaI} M. Bhargava, \textit{Higher composition laws I: A new view on Gauss composition, and quadratic generalizations}, Ann. of Math. (2) 159 (2004), no. 1, 217--250.

\bibitem {BhargavaII} M. Bhargava, \textit{Higher composition laws II: On cubic analogues of Gauss composition},
    Ann. Math. (2) 159 (2004), no. 2, 864--886.

\bibitem {BhargavaIII} M. Bhargava, \textit{Higher composition laws III: The parametrization of quartic rings}, Ann. Math. (2) 159 (2004), no. 3, 1329--1360.

\bibitem {Buell} D. A. Buell, \textit{Ideal composition in
quadratic fields: from Bhargava to Gauss}, Ramanujan J.
29 (2012), no. 1--3, 31--49.

\bibitem {Cox} D. A. Cox, \textit{Primes of the Form $x^2+ny^2$: Fermat, Class field theory, and Complex Multiplication}, 2nd edn, Pure and Applied Mathematics (Hoboken), John Wiley \& Sons, Inc., Hoboken, NJ, 2013.

\bibitem {D-L-R} H. Darmon, A. Lauder and V. Rotger,
\textit{Overconvergent generalised eigenforms of weight one and class fields of real quadratic fields},
Adv. Math. 283 (2015), 130--142.

\bibitem {Gauss} C. F. Gauss, \textit{Disquisitiones Arithmeticae}, Leipzig, 1801.

\bibitem {Hasse} H. Hasse, \textit{Neue Begr\"{u}ndung der Komplexen Multiplikation I, II},
J. f\"{u}r die Reine und Angewandte Math. 157, 165 (1927, 1931), 115--139, 64--88.

\bibitem {Hecke1912} E. Hecke, \textit{H\"{o}here Modulfunktionen und ihre Anwendung auf die Zahlentheorie},
Math. Ann. 71 (1912), 1--37.

\bibitem {Hecke1913} E. Hecke,
\textit{\"{U}ber die Konstruktion relativ-abelscher Zahlk\"{o}rper
duch Modulfunktionen von zwei Variablen}, Math. Ann. 74 (1913), 465--510.

\bibitem {Hungerford} T. W. Hungerford, \textit{Algebra}, Grad. Texts in Math. 73,
Springer-Verlag, New York-Berlin, 1980.

\bibitem {Hurwitz} A. Hurwitz, \textit{Mathematische Werke, Bd. II: Zahlentheorie, Algebra und Geometrie},
Birkh\"{a}user Verlag, Basel-Stuttgart, 1963.

\bibitem {Janusz} G. J. Janusz, \textit{Algebraic Number Fields},
2nd edn, Grad. Studies in Math. 7, Amer. Math. Soc., Providence,
R. I., 1996.

\bibitem {J-K-S} H. Y. Jung, J. K. Koo and D. H. Shin,
\textit{On some extension of Gauss' work and applications},
submitted, https://arxiv.org/abs/1810.06197.

\bibitem {K-L} D. Kubert and S. Lang, \textit{Modular Units}, Grundlehren der mathematischen Wissenschaften 244, Spinger-Verlag, New York-Berlin, 1981.

\bibitem {Lang} S. Lang, \textit{Algebraic Number Theory}, 2nd edn, Grad. Texts in Math. 110, Spinger-Verlag, New York, 1994.

\bibitem {Mastropietro} M. W. Mastropietro, \textit{Quadratic forms and relative quadratic extensions},
Thesis (Ph.D.), University of California, San Diego, 2000.

\bibitem {Mazur} B. Mazur, \textit{How can we construct abelian Galois extensions of basic number fields?}, Bull. Amer. Math. Soc. (N.S.) 48 (2011), no. 2, 155--209.

\bibitem {Milne} J. S. Milne, \textit{Algebraic Number Theory}, http://www.jmilne.org/math/CourseNotes/ANT.pdf.

\bibitem {Narkiewicz} W. Narkiewicz, \textit{Elementary and Analytic Theory of Algebraic Numbers}, Third edition, Springer Monographs in Math., Springer-Verlag, Berlin, 2004.

\bibitem {Ramachandra}
K. Ramachandra, \textit{Some applications of Kronecker's limit
formula}, Ann. of Math. (2) 80 (1964), 104--148.

\bibitem {Ribet} K. A. Ribet, \textit{A modular construction of unramified $p$-extensions of $\mathbb{Q}(\mu_p)$},
Invent. Math. 34 (1976), no. 3, 151--162.

\bibitem {Schappacher} N. Schappacher,
\textit{On the history of Hilbert's twelfth problem: a comedy of errors},
Mat\'{e}riaux pour l'histoire des math\'{e}matiques au XX$^\mathrm{e}$ si\`{e}cle (Nice, 1996), 243--273,
S\'{e}min. Congr. 3, Soc. Math. France, Paris, 1998.

\bibitem {Shimura70} G. Shimura, \textit{On canonical models of arithmetic quotients of bounded symmetric domains},
Ann. of Math. 91 (1970), 144--222; II, 92 (1970), 528--549.

\bibitem {Shimura} G. Shimura, \textit{Abelian Varieties with Complex Multiplication and Modular Functions}, Princeton University Press, Princeton, NJ, 1998.

\bibitem {Shimura00} G. Shimura, \textit{Arithmeticity in the theory of automorphic forms}, Mathematical Surveys and Monographs, 82. Amer. Math. Soc., Providence, RI, 2000.

\bibitem {S-T} G. Shimura and Y. Taniyama,
\textit{Complex Multiplication of Abelian Varieties and Its Applications to Number Theory},
Publications of the Mathematical Society of Japan 6, Math. Soc. Japan, Tokyo, 1961.

\bibitem {Siegel} C. L. Siegel,
\textit{Lectures on Advanced Analytic Number Theory},
Notes by S. Raghavan, Tata Institute of Fundamental Research Lectures on Mathematics 23, Tata Institute of Fundamental Research, Bombay 1965.

\bibitem {Silverman} J. H. Silverman, \textit{Advanced Topics in the Arithmetic of Elliptic Curves}, Grad. Texts in Math. 151, Springer-Verlag, New York, 1994.

\bibitem {Zemkova} K. Zemkova, \textit{Composition of quadratic forms over number fields}, Master Thesis, Charles University, 2018.

\end{thebibliography}

\address{
Department of Mathematics\\
Dankook University\\
Cheonan-si, Chungnam 31116\\
Republic of Korea} {hoyunjung@dankook.ac.kr}
\address{
Department of Mathematical Sciences \\
KAIST \\
Daejeon 34141\\
Republic of Korea} {jkgoo@kaist.ac.kr}
\address{
Department of Mathematics\\
Hankuk University of Foreign Studies\\
Yongin-si, Gyeonggi-do 17035\\
Republic of Korea} {dhshin@hufs.ac.kr}
\address{
Department of Mathematics Education\\
Pusan National University\\
Busan 46241\\Republic of Korea}
{dsyoon@pusan.ac.kr}

\end{document}